\documentclass[a4paper, 12pt]{amsart}
\usepackage{a4wide} 
\usepackage[english]{babel} 
\usepackage[T1]{fontenc} 
\usepackage{amsfonts} 
\usepackage{amsmath} 
\usepackage{amsthm} 
\usepackage{calrsfs} 
\usepackage{stmaryrd} 
\usepackage{graphicx} 
\usepackage{biblatex} 
\usepackage{csquotes} 
\usepackage{tikz} 

\DeclareMathAlphabet{\pazocal}{OMS}{zplm}{m}{n} 

\numberwithin{equation}{section}

\newtheorem{proposition}{Proposition}[section]
\newtheorem{lemma}{Lemma}[section]

\newtheorem{theorem}{Theorem}
\newtheorem{corollary}{Corollary}[section]
\newtheorem{definition}{Definition}[section]

\newtheorem{remark}{Remark}[section]

\newcommand{\C}{\mathbb{C}} 
\newcommand{\R}{\mathbb{R}} 
\newcommand{\Z}{\mathbb{Z}} 
\newcommand{\N}{\mathbb{N}} 
\newcommand{\B}{\mathbb{B}} 
\newcommand{\Sbb}{\mathbb{S}} 

\newcommand{\leb}{\mathsf{L}} 
\newcommand{\sob}{\mathsf{H}} 
\newcommand{\loc}{\mathrm{loc}} 
\newcommand{\comp}{\mathrm{comp}} 

\newcommand{\Ran}{\mathrm{Ran}} 
\newcommand{\Tr}{\mathrm{Tr}} 
\newcommand{\Sp}{\mathrm{Sp}} 
\newcommand{\disc}{\mathrm{disc}} 
\newcommand{\ess}{\mathrm{ess}} 
\newcommand{\Res}{\mathrm{Res}} 

\newcommand{\lp}{\left(} 
\newcommand{\rp}{\right)} 
\newcommand{\lb}{\left[} 
\newcommand{\rb}{\right]} 
\newcommand{\lcb}{\left\lbrace} 
\newcommand{\rcb}{\right\rbrace} 
\newcommand{\la}{\left\langle} 
\newcommand{\ra}{\right\rangle} 
\newcommand{\lv}{\left\vert} 
\newcommand{\rv}{\right\vert} 
\newcommand{\lV}{\left\Vert} 
\newcommand{\rV}{\right\Vert} 

\newcommand{\indic}{\mathbf{1}} 

\newcommand{\Dvec}{\mathbf{D}}

\newcommand{\nvec}{\mathbf{n}}
\newcommand{\Nvec}{\mathbf{N}}
\newcommand{\tvec}{\mathbf{t}}
\newcommand{\svec}{\mathbf{s}}

\newcommand{\vvec}{\mathbf{v}}

\newcommand{\xvec}{\mathbf{x}}
\newcommand{\yvec}{\mathbf{y}}

\newcommand{\nulvec}{\mathbf{0}}

\newcommand{\Acal}{\mathcal{A}}
\newcommand{\Ccal}{\mathcal{C}}
\newcommand{\Dcal}{\mathcal{D}}
\newcommand{\Fcal}{\mathcal{F}}
\newcommand{\Lcal}{\mathcal{L}}
\newcommand{\Scal}{\mathcal{S}}

\newcommand{\Apaz}{\pazocal{A}}

\newcommand{\Mpaz}{\pazocal{M}}

\newcommand{\Rpaz}{\pazocal{R}}

\newcommand{\Tpaz}{\pazocal{T}}
\newcommand{\Upaz}{\pazocal{U}}
\newcommand{\Vpaz}{\pazocal{V}} 
\newcommand{\Wpaz}{\pazocal{W}}

\newcommand{\Zpaz}{\pazocal{Z}}

\newcommand{\Hfrak}{\mathfrak{H}}
\newcommand{\Dfrak}{\mathfrak{D}}

\newcommand{\mfrak}{\mathfrak{m}}
\newcommand{\pfrak}{\mathfrak{p}}

\DeclareMathOperator{\Supp}{\mathrm{Supp}} 
\DeclareMathOperator{\tq}{\mathrm{ | }} 
\newcommand{\Ind}{\mathrm{Ind}} 

\title{Absence of real resonances of Dirac operators}
\date{\today}
\author{Henry DUMANT}
\address{Univ. Bordeaux, CNRS, Bordeaux INP, IMB, UMR 5251, F-33400 Talence, France }
\email{henry.dumant@math.u-bordeaux.fr}
\begin{document}

\begin{abstract} 
    The purpose of this paper is to introduce the resonances of tridimensional Dirac operators by continuing meromorphically the resolvent and to establish a result about their localization : a kind of Rellich Theorem. We first consider the case of the Dirac operator in an external field which is essentially bounded and compactly supported. Secondly, we consider the case of the MIT bag model outside a smooth and bounded obstacle.
\end{abstract}

\maketitle

\section{Introduction}
Scattering resonances are the replacement of discrete spectral data for problems on non-compact domains. Several approaches have been developped for their study in the case of Schrödinger operators. For Dirac operators, most of the contributions are based on the method of \textit{complex scaling} that use mainly analytic microlocal analysis (see \cite{AgCo71,Se88}). With this method, a semi-classical limit of the resonances is studied in \cite{Pa91,Pa92} and a relativistic limit of the resonances is studied in \cite{AmBrNo01}. A \textit{limiting absorption principle}, which can be related to resonances, has also been shown in \cite{BaHe92}. The first purpose of this work is to adapt the \textit{potential and obstacle scattering}, formalized for Schrödinger operators in \cite[Theorem 1.1]{SjZw91} (see also \cite{Sj02,DyZw19}), to Dirac operators. We mention that, for Dirac operators, the potential scattering has been introduced for unidimensional and radial cases in \cite{IaKo14,IaKo15} or for the study of Poisson wave formulas in \cite{KuMe17,ChMe21}. The principle starts with the fact that if $L$ is an unbounded self-adjoint operator in a Hilbert space $\Hfrak$ then the resolvent function $z\mapsto (L-z)^{-1}$ is meromorphic from $\C\setminus\Sp_\ess(L)$ to $\Lcal(\Hfrak)$ and that the set of its poles is $\Sp_\disc(L)$. One idea is to encode the resonances as poles of the meromorphic continuation to an approriate Riemann surface of the resolvent, seen in suitable functional spaces. Such an idea first appeared in \cite{DoMcLTh66}. Then a natural question is the localization of the resonances in the surface. One of the classic results for Schrödinger operators is that a resonance may not live in the essential spectrum, result that sometimes wear the name of Rellich uniqueness Theorem (see \cite[Theorem 3.33, 3.35]{DyZw19}). In the semi-classical regime, much precise results mention that there are no resonances in a neighbourhood of the essential spectrum (see \cite{Bu98,Vo14,KlVo19} as well as the works cited therein). The second purpose of this work is to prove the Rellich uniqueness Theorem for some Dirac operators.  

\subsection{Preliminaries}
\subsubsection{Tridimensionnal massive Dirac operator} 
The three $2\times 2$ hermitian matrices
\[\sigma_1:=\begin{pmatrix} 0 & 1 \\ 1 & 0 \end{pmatrix} \mbox{, } \sigma_2:=\begin{pmatrix} 0 & -\imath \\ \imath & 0 \end{pmatrix} \mbox{ and } \sigma_3:=\begin{pmatrix} 1 & 0 \\ 0 & -1 \end{pmatrix}\]
are called the Pauli matrices. The four $4\times 4$ hermitian matrices
\[\beta:=\begin{pmatrix} I_2 & 0_2 \\ 0_2 & -I_2 \end{pmatrix} \mbox{ and } \alpha_j:=\begin{pmatrix} 0_2 & \sigma_j \\ \sigma_j & 0_2 \end{pmatrix} \mbox{ for } j\in\llbracket 1,3\rrbracket\]
are called the Dirac matrices. 
Formally, if $X:=(X_1,X_2,X_3)$ 
\[\alpha\cdot X:=\alpha_1X_1+\alpha_2X_2+\alpha_3X_3.\]
The Dirac operator associated with the energy of a particle living in $\R^3$, of mass $m\in\R$ and of spin $1/2$ is the first order differential operator $m\beta+\alpha\cdot D$ where we recall that $D_j=-\imath\partial_j$ for $j\in\llbracket 1,3\rrbracket$. It acts on four components distributions. In this paper, we will only discuss the case $m=1$ since a correct transformation permits to obtain all our results when $m\not=0$. A study of the case $m=0$ may also be interesting (see \cite{IaKo14} for this direction). To allege the reading, we set
\[\Dvec:=\beta+\alpha\cdot D.\]

\subsubsection{Fundamental solution} 
Let $z\in\C$. Using standard anti-commutation relations between the Dirac matrices (see \cite[Appendix 1.B]{Th91}), one computes 
\begin{equation}\label{formal_square_dirac}
    (\Dvec-zI_4)(\Dvec+zI_4)=-\Delta-(z^2-1)I_4.
\end{equation}
Consider $\omega\in\C$ such that $\omega^2=z^2-1$. It is well known that a fundamental solution $G_\omega$ of the Helmoltz operator $-\Delta-\omega^2$ can be defined by
\[G_\omega(\xvec)=\frac{e^{\imath\omega\lv\xvec\rv}}{4\pi\lv\xvec\rv} \mbox{ for } \xvec\not=\nulvec.\]
Considering \eqref{formal_square_dirac}, the distribution
\[F_{z,\omega}:=(\Dvec+z)(G_\omega I_4)\]
is a fundamental solution of $\Dvec-z$. For $\xvec\not=\nulvec$, one has 
\begin{equation}\label{DiracKernel}
    F_{z,\omega}(\xvec)=\frac{e^{\imath\omega\lv\xvec\rv}}{4\pi\lv\xvec\rv}\lp zI_4+\beta+(\imath+\omega\lv\xvec\rv)\alpha\cdot\frac{\xvec}{\lv\xvec\rv^2}\rp.
\end{equation}

\subsubsection{Dirac operator in an external field}\label{PotentialContext}
We will first perturb $\Dvec$ by an external field $V$  which may be think as an electric or a magnetic potential.

The operator $H_0:=\Dvec$ on the domain $\Dfrak:=\sob^1(\R^3)^4$ is self-adjoint in the ambient Hilbert space $\Hfrak:=\leb^2(\R^3)^4$, its spectrum is purely continuous and
\[\Sp(H_0)=]-\infty,-1]\cup[1,+\infty[\]
(see \cite[Theorem 1.1]{Th91}).
For $\phi\in\Hfrak$ and $z\not\in\Sp(H_0)$ one has
\begin{equation}\label{convol_res_libre}
    \lp H_0-z\rp^{-1}\phi=F_{z,(z^2-1)^{1/2}}*\phi
\end{equation}
where the power $1/2$ denotes the branch of the square root defined on $\C\setminus\left[0,+\infty\right[$ and valued in $\lcb z\in\C\tq\Im z>0\rcb$. 

The potential $V$ is a $4\times 4$ matrix which entries are compactly supported functions of $\leb^\infty(\R^3)$. One shows that the operator $H:=H_0+V$ on the domain $\Dfrak$ is closed and thanks to Weyl's Theorem
\[\Sp_\ess(H)=]-\infty,-1]\cup [1,+\infty[.\]
Rellich-Kato Theorem ensures that $H$ is self-adjoint when $V^*=V$.

\subsubsection{Dirac operator with MIT boundary condition}\label{ObstacleContext} 
In a second time, we will perturb $\Dvec$ by an obstacle. We impose the so-called MIT boundary condition at the boundary of the obstacle. Such a model is used to study confined particles of spin $1/2$ (see \cite{ChJaJoTh74,ChJaJoThWe74}) or is considered in general relativity (see \cite{Ba08}).

\begin{figure}
    \centering
    \begin{tikzpicture}
        \fill[gray!20] 
          (0,0) 
          .. controls (1,2.5) and (3,3) .. (4,2)
          .. controls (5,1) and (6,3) .. (7,2.5)
          .. controls (8,2) and (8,-2) .. (7,-2.5)
          .. controls (6,-3) and (5,-1.5) .. (4,-2)
          .. controls (3,-2.5) and (2,-3) .. (1,-2.5)
          .. controls (0,-2) and (-0.2,-0.5) .. (0,0);
        
        \draw[thick] 
          (0,0) 
          .. controls (1,2.5) and (3,3) .. (4,2)
          .. controls (5,1) and (6,3) .. (7,2.5)
          .. controls (8,2) and (8,-2) .. (7,-2.5)
          .. controls (6,-3) and (5,-1.5) .. (4,-2)
          .. controls (3,-2.5) and (2,-3) .. (1,-2.5)
          .. controls (0,-2) and (-0.2,-0.5) .. (0,0);

        \draw[->] (7.15,2.4) -- (6.5,1.7);
        
        \node at (-1.2,1) {$\Omega$};
        \node at (3.7,0) {$\R^3\setminus\Omega$};
        \node at (4.2,2.25) {$\partial\Omega$};
        \node at (7.0,1.8) {$\nvec$};
    \end{tikzpicture}
    \caption{}
    \label{FigureObstacle}
\end{figure}

The obstacle is $\R^3\setminus\Omega$ where $\Omega\subset\R^3$ is open, $\Ccal^2$-smooth and such that $\R^3\setminus\Omega$ is bounded. The outward pointing normal to the boundary of $\Omega$ is denoted $\nvec$ (see Figure \ref{FigureObstacle}). If $\psi\in\sob^1(\Omega)^4$ then $\psi_{\tq\partial\Omega}\in\leb^2(\partial\Omega)^4$ denotes its trace. The operator $H:=\Dvec$ on its domain 
\[\Dfrak:=\lcb\psi\in\sob^1(\Omega)^4\tq (-\imath\beta\alpha\cdot\nvec)\psi_{\tq\partial\Omega}=\psi_{\tq\partial\Omega}\rcb\]
 is self-adjoint in the ambient Hilbert space $\Hfrak:=\leb^2(\Omega)^4$ (see \cite[Theorem 3.2]{OuVe18}). Moreover
\[\Sp_\ess(H)=]-\infty,-1]\cup[1,+\infty[\]
(see \cite[Theorem 3.1]{BeBrZr24}). For all $\psi\in\Dfrak$
\begin{equation}\label{EquivalentNormMIT}
    \frac{1}{C}\lV H\psi\rV_{\leb^2(\Omega)^4}\leq \lV\psi\rV_{\sob^1(\Omega)^4}\leq C\lV H\psi\rV_{\leb^2(\Omega)^4}
\end{equation}
where $C>0$ does not depend on $\psi$. Since $H$ is closed on its domain we deduce that $\Dfrak$ is a Banach space when equipped with $\lV\cdot\rV_{\Dfrak}:=\lV\cdot\rV_{\sob^1(\Omega)^4}$. For completeness, \eqref{EquivalentNormMIT} is proven in the appendix.

\subsection{Statement of the results}\label{Statement} 
The first aim of the paper is to extend meromorphically the resolvent of $H$ across $\Sp_\ess(H)$. Consider the easiest case which is the one of the free Dirac operator. A natural way to continue the resolvent of $H_0$ is to use its convolution representation given by \eqref{convol_res_libre}. The question reduces to the continuation of $z\mapsto(z^2-1)^{1/2}$ across $(-\infty,-1]\cup[1,+\infty)$. Therefore we need the Riemann surface associated with this function. Specificly, we introduce
\[\Mpaz:=\lcb (z,\omega)\in\C^2\tq z^2-\omega^2=1\rcb\]
which we equip with an appropriate holomorphic atlas (the details are given in Section \ref{AnalyticalContext}). If $(z,\omega)\in\Mpaz$ with $\Im\omega>0$ then $z\not\in\Sp(H_0)$ and $(H_0-z)^{-1}$ can be considered as a function of $(z,\omega)$. Specificly for $\phi\in\Hfrak$
\[(H_0-z)^{-1}\phi=F_{z,\omega}*\phi.\]
When $(z,\omega)\in\Mpaz$ with $\Im\omega\leq 0$ we would like to define the continuation of $(H_0-z)^{-1}\phi$ as the convolution between $F_{z,\omega}$ and $\phi$. However, $F_{z,\omega}$ being only locally square integrable in that case, we need $\phi$ to be compactly supported.
The example of $H_0$ makes it clear that one needs to suit approriate functional spaces in order to continue the resolvent. For the operators $H$, we will consider 
\[\Hfrak_\comp=\lcb\phi\in\Hfrak\tq\Supp\phi \mbox{ is bounded}\rcb \mbox{ and } \Dfrak_\loc=\lcb\psi\in\Dcal'\tq\chi\psi\in\Dfrak\mbox{ for all } \chi\in\Ccal\rcb\]
where $\Dcal':=\Dcal'(\R^3)^4$, $\Ccal:=\Ccal^\infty_c(\R^3)$ in the context \ref{PotentialContext} and $\Dcal':=\Dcal'(\Omega)^4$, $\Ccal:=\Ccal^\infty_c(\overline{\Omega})$ in then context \ref{ObstacleContext}. One defines $\Lcal(\Hfrak_\comp,\Dfrak_\loc)$ as the vector space of the linear functions $T:\Hfrak_\comp\rightarrow\Dfrak_\loc$ such that $\chi T\tilde{\chi}$ is bounded from $\Hfrak$ to $\Dfrak$ for all $\chi,\tilde{\chi}\in\Ccal$.

\bigskip

We are able to state the first result of this work.

\begin{theorem}[Meromorphic continuations, resonances]\label{ResolventContinued} 
    Consider both contexts \ref{PotentialContext} and \ref{ObstacleContext}. There exists a finite meromorphic function from $\Mpaz$ to $\Lcal(\Hfrak_\comp,\Dfrak_\loc)$ denoted $\Rpaz$ such that if $(z,\omega)\in\Mpaz$ is not a pole of $\Rpaz$ and $\Im\omega>0$ then $z\not\in\Sp(H)$ and
    \[\Rpaz(z,\omega)=(H-z)^{-1} \mbox{ on } \Hfrak_\comp.\]
    A pole of $\Rpaz$ is called a resonance of $H$.
\end{theorem}

The notion of finite meromorphic function from $\Mpaz$ to $\Lcal(\Hfrak_\comp,\Dfrak_\loc)$ is precised in Section \ref{AnalyticalContext}. Although we focus ourselves on tridimensional Dirac operators, it should be possible to adapt our results in other dimensions once the Riemann surface involved is well understood (in even dimension, logarithmic singularities should appeared at the ramification points). Unlike the case of Schrödinger operators where one omits the abstract setting of the Riemann surface by parametrizing it with $\omega\mapsto\omega^2$, this setting is needed for Dirac operators as $z$ and $\omega$ appear simultaneously in \eqref{DiracKernel}. Moreover, it permits to work with a global definition that avoids treating cases according to the localization on the surface as in \cite{KuMe17,ChMe21}. 

\bigskip

We are now ready to state the major result of our paper. For Schrödinger operators, such a result is known as Rellich uniqueness Theorem (see \cite{Sj02,DyZw19}).

\begin{theorem}[Absence of real resonances]\label{Rellich}
    Let $(\lambda,\kappa)\in\Mpaz$ with $\kappa\in\R$.
    \begin{enumerate}
        \item In the context \ref{PotentialContext}, if $V^*=V$ and $\kappa\not=0$ then $(\lambda,\kappa)$ is not a resonance of $H$.
        \item In the context \ref{ObstacleContext},
        \begin{itemize}
            
            \item if $\kappa=0$ then $(\lambda,\kappa)$ is not a resonance of $H$.
            \item if $\Omega$ is connected and $\kappa\not=0$ then $(\lambda,\kappa)$ is not a resonance of $H$.
        \end{itemize}
    \end{enumerate}
\end{theorem}

\begin{remark}
    In the context \ref{PotentialContext}, $H$ can have an eigenvalue at $\pm1$ (see \cite{El00}). This explains the difference between the two assertions of Theorem \ref{Rellich}.
\end{remark}

\subsection{Structure of the paper} 
This paper is organized as follows. In Section \ref{free_scattering}, we study the holomorphic continuation of the free resolvent and provide some asymptotics in the space variable. In Section \ref{PotentialScattering}, we establish Theorem \ref{ResolventContinued} and \ref{Rellich} for the context \ref{PotentialContext} whereas the context \ref{ObstacleContext} is treated in Section \ref{ObstacleScattering}. Some parts of the proofs for the obstacle's case uses the ones of the potential's case. 

\section{The analytical setup}\label{AnalyticalContext}

In this section, we introduce the appropriate notions of analyticity needed. We prefer to be specific since our functions will be valued in $\Lcal(\Hfrak_\comp,\Dfrak_\loc)$ which is not a Banach space. Nevertheless we come back to $\Lcal(\Hfrak,\Dfrak)$ that has the Banach structure.

\bigskip

We start with the functions defined in an open subset of $\C$ and valued in $\Lcal(\Hfrak_\comp,\Dfrak_\loc)$. 

\begin{definition}\label{AnalyticC}
    Let $\Tpaz:\Wpaz\rightarrow\Lcal(\Hfrak_\comp,\Dfrak_\loc)$ with $\Wpaz\subset\C$ open. One says that $\Tpaz$ is holomorphic if $\zeta\mapsto\chi\Tpaz(\zeta)\tilde{\chi}$ is holomorphic from $\Wpaz$ to $\Lcal(\Hfrak,\Dfrak)$ for all $\chi,\tilde{\chi}\in\Ccal$.
\end{definition}

Now we discuss about the complex structure on $\Mpaz$. Thanks to the holomorphic implicit function Theorem, for all $(z_0,\omega_0)\in\Mpaz$ there exists an open subset $\Vpaz_{(z_0,\omega_0)}$ of $\Mpaz$ containing $(z_0,\omega_0)$ and an open subset $\Wpaz_{(z_0,\omega_0)}$ of $\C$ such that
\[\varphi_{(z_0,\omega_0)}:(z,\omega)\mapsto\left\{
    \begin{array}{ll}
        z \mbox{ if } \omega_0\not=0 \\
        \omega \mbox{ otherwise.} 
    \end{array}
\right.\]
is invertible from $\Vpaz_{(z_0,\omega_0)}$ to $\Wpaz_{(z_0,\omega_0)}$ with $\varphi_{(z_0,\omega_0)}^{-1}$ being holomorphic from $\Wpaz_{(z_0,\omega_0)}$ to $\C^2$. One checks that $\lcb(\Vpaz_\mfrak,\varphi_\mfrak)\tq\mfrak\in\Mpaz\rcb$ is an holomorphic atlas on $\Mpaz$. 
We equip $\Mpaz$ with the maximal holomorphic atlas $\Acal$ associated with $\lcb(\Vpaz_\mfrak,\varphi_\mfrak)\tq\mfrak\in\Mpaz\rcb$. Specificly $\Acal$ is the set of the local chart $(\Vpaz,\varphi)$ on $\Mpaz$ such that $\varphi\circ\varphi_\mfrak^{-1}$ is holomorphic on $\varphi_\mfrak(\Vpaz\cap\Vpaz_\mfrak)$ for all $\mfrak\in\Mpaz$ such that $\Vpaz\cap\Vpaz_\mfrak\not=\emptyset$.

\bigskip 

We can now define what is a holomorphic function defined in an open subset of $\Mpaz$ and valued in $\Lcal(\Hfrak_\comp,\Dfrak_\loc)$. 

\begin{definition}\label{AnalyticM}
    Let $\Tpaz:\Upaz\rightarrow\Lcal(\Hfrak_\comp,\Dfrak_\loc)$ with $\Upaz\subset\Mpaz$ open. One says that $\Tpaz$ is holomorphic if $\Tpaz\circ\varphi^{-1}$ is holomorphic from $\varphi(\Upaz\cap\Vpaz)$ to $\Lcal(\Hfrak_\comp,\Dfrak_\loc)$ for all $(\Vpaz,\varphi)\in\Acal$.
\end{definition}

We finally define a meromorphic function defined in $\Mpaz$ and valued in $\Lcal(\Hfrak_\comp,\Dfrak_\loc)$. 

\begin{definition}\label{Meromorphic}
    One calls finite meromorphic function from $\Mpaz$ to $\Lcal(\Hfrak_\comp,\Dfrak_\loc)$ any data $(\Zpaz,\Tpaz)$ such that
    \begin{enumerate}
        \item $\Zpaz$ is closed and discrete in $\Mpaz$
        \item $\Tpaz$ is holomorphic from $\Mpaz\setminus\Zpaz$ to $\Lcal(\Hfrak_\comp,\Dfrak_\loc)$
        \item\label{FiniteLaurent} for all $\mfrak\in\Zpaz$ there exists $(\Vpaz,\varphi)\in\Acal$ with $\mfrak\in\Vpaz$ and $T_1,...,T_L\in\Lcal(\Hfrak_\comp,\Dfrak_\loc)$ of finite rank such that 
        \[\zeta\mapsto\Tpaz(\varphi^{-1}(\zeta))-\sum_{\ell=1}^L\frac{T_\ell}{(\zeta-\varphi(\mfrak))^\ell}\]
        can be continued as an holomorphic function from a open neighbourhood of $\varphi(\mfrak)$ to $\Lcal(\Hfrak_\comp,\Dfrak_\loc)$.
    \end{enumerate}
    One calls $T_1$ the residue at $\mfrak$ of $\Tpaz$ in $(\Vpaz,\varphi)$. Any $\mfrak\in\Zpaz$ such that \eqref{FiniteLaurent} holds with $T_L\not=0$ is called a pole of $\Tpaz$.
\end{definition}

\begin{remark}\label{FollowCutoff}
    The set $\Lcal(\Hfrak_\comp,\Dfrak_\loc)$ does not benefit of a natural algebra structure. We shall check by hand that compositions are allowed and that the analytical properties are preserved. Since we want a weak dependance in the cut-offs, it will be crucial to follow them precisely in our proofs.
\end{remark}

\section{The free outgoing scattering resolvent}\label{free_scattering}
We start with the case of the free Dirac operator. Let us recall that, in this section, $\Hfrak=\leb^2(\R^3)^4$, $\Dfrak=\sob^1(\R^3)^4$ and $\Ccal=\Ccal^\infty_c(\R^3)$.
Thanks to the discussion at the begining of Subsection \ref{Statement}, we set 
for $\mfrak\in\Mpaz$ and $\phi\in\Hfrak_\comp$
\[\Rpaz_0(\mfrak)\phi:=F_\mfrak*\phi.\]
Theorem \ref{HolomorphyR0} is about the holomorphy of $\Rpaz_0$ whereas some crucial asymptotics are given in Proposition \ref{FreeResolventAsymptotics}.

\begin{theorem}\label{HolomorphyR0}
    The function $\Rpaz_0$ is holomorphic from $\Mpaz$ to $\Lcal(\Hfrak_\comp,\Dfrak_\loc)$. 
    
    Specificly, for all $(\Vpaz,\varphi)\in\Acal$ and all $\zeta_0\in\varphi(\Vpaz)$ there exists a sequence $(A_\ell)_{\ell\geq 0}$ of $\Lcal(\Hfrak_\comp,\Dfrak_\comp)$ such that for all $\chi,\tilde{\chi}\in\Ccal$ and all $\zeta$ in a neighbourhood of $\zeta_0$ 
    \[\chi \Rpaz_0(\varphi^{-1}(\zeta))\tilde{\chi}=\sum_{\ell=0}^{+\infty}(\zeta-\zeta_0)^\ell\chi A_\ell\tilde{\chi}\]
    where the convergence of the series holds in $\Lcal(\Hfrak,\Dfrak)$.
\end{theorem}

\begin{remark} 
    In other words, $H_0$ does not have resonances. Also, we hilight the dependance in the cut-offs, as explained in Remark \ref{FollowCutoff}.
\end{remark}

We will need the next Lemma for the proof of Theorem \ref{HolomorphyR0}.

\begin{lemma}\label{EllipticDirac}
    Let $\psi\in\Hfrak$ such that $\Dvec\psi\in\Hfrak$. Then $\psi\in\Dfrak$ and
    \[\lV\psi\rV_{\Dfrak}=C\lV\Dvec\psi\rV_{\Hfrak}\]
    with $C>0$ independent of $\psi$.
\end{lemma}

\begin{proof} Since for $\xi\in\R^3$
    \[(\beta+\alpha\cdot\xi)^*(\beta+\alpha\cdot\xi)=(1+\lv\xi\rv^2)I_4\]
    one has
    \begin{align*}
        \int_{\R^3}\lv\Fcal(\Dvec\psi)(\xi)\rv^2d\xi&=\int_{\R^3}\lv(\beta+\alpha\cdot\xi)\Fcal\psi(\xi)\rv^2d\xi\\
        &=\int_{\R^3}\lv\la\xi\ra\Fcal\psi(\xi)\rv^2d\xi
    \end{align*}
    so the conclusion follows from Plancherel's Theorem.
\end{proof}

We can now show Proposition \ref{HolomorphyR0}.

\begin{proof} 
    For the reading we set
    \[R_0(\zeta):=\chi \Rpaz_0(\varphi^{-1}(\zeta))\tilde{\chi} \mbox{ , } f_\zeta:=F_{\varphi^{-1}(\zeta)} \mbox{ for } \zeta\in\Wpaz:=\varphi(\Vpaz).\]
    By maximality of the atlas $\Apaz$, $\varphi^{-1}$ is holomorphic from $\Wpaz$ to $\C^2$. Let us introduce, for $\xvec\not=\nulvec$ and $\ell\geq 0$
    \[a_\ell(\xvec):=\frac{1}{\ell!}\frac{d^\ell}{d\zeta^\ell}\lp f_\zeta(\xvec)\rp_{\tq\zeta=\zeta_0}.\]
    It is clear that each entry of $a_\ell$ is a function of $\leb^1_\loc(\R^3)$. We denote by $A_\ell$ the convolution operator by $a_\ell$. The Young estimate for the convolution ensures that $R_0(\zeta)$ and $\chi A_\ell\tilde{\chi}$ are bounded operators of $\Hfrak$ and for all $L\geq 0$ and $\zeta\in\Wpaz$
    \begin{equation}\label{AnalyticEstimateR0}
        \lV R_0(\zeta)-\sum_{\ell=0}^{L}(\zeta-\zeta_0)^\ell\chi A_\ell\tilde{\chi}\rV_{\Lcal(\Hfrak)}\leq C\int_{K}\lv f_\zeta(\xvec)-\sum_{\ell=0}^L(\zeta-\zeta_0)^\ell a_\ell(\xvec)\rv d\xvec
    \end{equation}
    where $C$ is a positive constant and $K$ is a compact that only depend on the supports of $\chi$ and $\tilde{\chi}$. The upper bound from \eqref{AnalyticEstimateR0} goes to zero as $L$ goes to infinity for all $\zeta$ in a neighbourhood of $\zeta_0$. For those $\zeta$
    \begin{equation}\label{AnalyticityR0}
        R_0(\zeta)=\sum_{\ell=0}^{+\infty}(\zeta-\zeta_0)^\ell\chi A_\ell\tilde{\chi}
    \end{equation}
    where the convergence holds in $\Lcal(\Hfrak)$. 
    
    Until the end of the proof, we wish to show that \eqref{AnalyticityR0} holds in $\Lcal(\Hfrak,\Dfrak)$. For $\zeta\in\Wpaz$, since $f_\zeta$ is a fundamental solution of $\Dvec-z(\zeta)I_4$ one has
    \begin{equation}\label{inverse_tronquée}
        \Dvec R_0(\zeta)=\chi\tilde{\chi}+z(\zeta)R_0(\zeta)+\alpha\cdot(D\chi)\Rpaz_0(\zeta)\tilde{\chi}.
    \end{equation} 
    Thanks to \eqref{inverse_tronquée} and Lemma \ref{EllipticDirac}, $R_0(\zeta)$ is bounded from $\Hfrak$ to $\Dfrak$ for any $\zeta\in\Wpaz$. Let us write 
    \[B_\ell:=\frac{1}{\ell!}\frac{d^\ell}{d\zeta^\ell}(H_0 R_0(\zeta))_{\tq\zeta=\zeta_0}.\]
    First $\chi A_0\tilde{\chi}=R_0(\zeta_0)\in\Lcal(\Hfrak,\Dfrak)$ and $H_0\chi A_0\tilde{\chi}=B_0$. For all $\psi\in\Hfrak$
    \[\lim_{\zeta\rightarrow\zeta_0}\frac{R_0(\zeta)-R_0(\zeta_0)}{\zeta-\zeta_0}\psi=(\chi A_1\tilde{\chi})\psi \mbox{ and } \lim_{\zeta\rightarrow\zeta_0} H_0\lp\frac{R_0(\zeta)-R_0(\zeta_0)}{\zeta-\zeta_0}\psi\rp=B_1\psi\]
    but $H_0$ being closed this yields $(\chi A_1\tilde{\chi})\psi\in\Dfrak$ and 
    \[H_0\lp(\chi A_1\tilde{\chi})\psi)\rp=B_1\psi.\]
    One deduces of Lemma \ref{EllipticDirac} that $\chi A_1\tilde{\chi}$ is bounded from $\Hfrak$ to $\Dfrak$. Iterating the argument, one obtains that $\chi A_\ell\tilde{\chi}\in\Lcal(\Hfrak,\Dfrak)$ and $H_0\chi A_\ell\tilde{\chi}=B_\ell$ for all $\ell\in\N$. Again by Lemma \ref{EllipticDirac} one has
    \[\lV R_0(\zeta)-\sum_{\ell=0}^L(\zeta-\zeta_0)^\ell\chi A_\ell\tilde{\chi}\rV_{\Lcal(\Hfrak,\Dfrak)}=
    C\lV H_0R_0(\zeta)-\sum_{\ell=0}^L(\zeta-\zeta_0)^\ell B_\ell\rV_{\Lcal(\Hfrak)}\]
    for all $L\geq 0$. Taking $L\rightarrow+\infty$ gives the result.
\end{proof}

\begin{proposition}\label{FreeResolventAsymptotics}
    Let $\phi\in\Hfrak_\comp$ and $(z,\omega)\in\Mpaz$. The function $\Rpaz_0(z,\omega)\phi$ is smooth on $\R^3\setminus\Supp\phi$. Moreover as $r\rightarrow+\infty$
    \[\Rpaz_0(z,\omega)\phi(r\svec)=\frac{e^{\imath\omega r}}{4\pi r}\lp\lp zI_4+\beta+\omega\alpha\cdot\svec\rp\int_{\R^3}e^{-\imath\omega\svec\cdot\xvec}\phi(\xvec)d\xvec+O\lp\frac{1}{r}\rp\rp\]
    \[D_j\Rpaz_0(\pm1,0)\phi(r\svec)=\frac{\imath s_j}{4\pi r^2}(\pm I_4+\beta)\int_{\R^3}\phi+O\lp\frac{1}{r^3}\rp \mbox{ , } j\in\llbracket 1,3\rrbracket\]
    uniformly in $\svec\in\R^3$ such that $\lv\svec\rv=1$.
\end{proposition}

\begin{proof} 
    The regularity result is a straight consequence of the dominated convergence Theorem. We move on the asymptotics. One easily computes, as $\lv\yvec\rv\rightarrow+\infty$ 
    \[F_{z,\omega}(\yvec)=\frac{e^{\imath\omega\lv\yvec\rv}}{4\pi\lv\yvec\rv}\lp zI_4+\beta+\omega\alpha\cdot\frac{\yvec}{\lv\yvec\rv}+O\lp\frac{1}{\lv\yvec\rv}\rp\rp\]
    and 
    \[D_jF_{\pm1,0}(\yvec)=\frac{1}{4\pi\lv\yvec\rv^2}\lp \frac{\imath y_j}{\lv\yvec\rv}(\pm I_4+\beta)+O\lp\frac{1}{\lv\yvec\rv}\rp\rp.\]
    Let $\xvec\in\Supp\phi$. As $r\rightarrow+\infty$, one has
    \[\lv r\svec-\xvec\rv=r-\svec\cdot\xvec+O\lp\frac{1}{r}\rp\]
    and hence
    \[\frac{1}{\lv r\svec-\xvec\rv}=\frac{1}{r}\lp 1+O\lp\frac{1}{r}\rp\rp \mbox{ and } e^{\imath\omega\lv r\svec-\xvec\rv}=e^{\imath\omega r}\lp e^{-\imath\omega\svec\cdot\xvec}+O\lp\frac{1}{r}\rp\rp.\]
    Since $\lv\svec\rv=1$ and $\xvec\in\Supp\phi$, all those asymptotics are uniform in $(\svec,\xvec)$. Therefore the asymptotics follow by integrating against $\phi$.
\end{proof}

\section{Potential scattering}\label{PotentialScattering}
In this section, we consider the context \ref{PotentialContext}. We construct a meromorphic continuation of the resolvent of $H$ and study the appearing singularities. The ideas and techniques used in the proofs of the following results are standard and attributed to  Sjöstrand, Zworski and Dyatlov \cite{Sj02,DyZw19}. For completeness, we give all the details, particularly concerning the cut-offs (see Remark \ref{FollowCutoff}). 

\bigskip

We recall that we still have $\Hfrak=\leb^2(\R^3)^4$, $\Dfrak=\sob^1(\R^3)^4$ and $\Ccal=\Ccal^\infty_c(\R^3)$.

\subsection{Resonances of $H$}
The next Theorem is a more precise version of Theorem \ref{ResolventContinued} in the context \ref{PotentialContext}.

\begin{theorem}\label{PotentialResolventContinued}
    There exists a finite meromorphic function from $\Mpaz$ to $\Lcal(\Hfrak_\comp,\Dfrak_\loc)$ denoted $(\Zpaz,\Rpaz)$ such that if $(z,\omega)\not\in\Zpaz$ and $\Im\omega>0$ then $z\not\in\Sp(H)$ and
    \[\Rpaz(z,\omega)=(H-z)^{-1} \mbox{ on } \Hfrak_\comp.\]
    Any element of $\Zpaz$ is a pole of $\Rpaz$ and
    \[\Zpaz=\lcb (\lambda,\kappa)\in\Mpaz\tq\exists\psi\in\Ran(\Rpaz_0(\lambda,\kappa)) \mbox{ s.t. } \psi\not=0 \mbox{ and }(\Dvec+V)\psi=\lambda\psi\rcb.\]
\end{theorem}

\begin{definition}
    The resonances of $H$ are the poles of $\Rpaz$. Their set is denoted $\Res(H)$. 
\end{definition}

\begin{remark}\label{RemarkResonantState}
    Associated with a resonance $(\lambda,\kappa)$, we recover a generalized eigenvalue problem according to the description of $\Zpaz$. In general, the corresponding function $\psi$ can not be square integrable. The integrability is replaced by the condition $\psi\in\Ran(\Rpaz_0(\lambda,\kappa))$. In particular, Proposition \ref{FreeResolventAsymptotics} provides a behaviour of $\psi$ at $\infty$. One says that $\psi$ is outgoing (as for Schrödinger operators). The Subsection \ref{PotentialResonantState} is dedicated to the study of such functions.
\end{remark}

Let us give the proof of Theorem \ref{PotentialResolventContinued}. The notation $f\prec g$ means that $g$ is equal to $1$ on the support of $f$.

\begin{proof} 
    Let $z\not\in\Sp(H)$. We start with 
    \begin{equation}\label{eq_res_dirac_élec_temp1}
        H-z=\lp 1+V(H_0-z)^{-1}\rp (H_0-z).
    \end{equation}
    Take $\rho\in\Ccal$ such that $V\prec\rho$. Then
    \begin{equation}\label{eq_res_dirac_élec_temp2}
        1+V(H_0-z)^{-1}=\lp 1+V(H_0-z)^{-1}(1-\rho)\rp\lp 1+V(H_0-z)^{-1}\rho\rp.
    \end{equation}
    But $1+V(H_0-z)^{-1}(1-\rho)$ and $1-V(H_0-z)^{-1}(1-\rho)$ are inverse of one another therefore
    \begin{equation}\label{eq_res_dirac_élec_fin}
        (H-z)^{-1}=(H_0-z)^{-1}\lp 1+V(H_0-z)^{-1}\rho\rp^{-1}\lp 1-V(H_0-z)^{-1}(1-\rho)\rp.
    \end{equation}
    Then it is natural to define $\Rpaz(\mfrak)$ by replacing $(H_0-z)^{-1}$ by $\Rpaz_0(\mfrak)$ in the r.h.s from \eqref{eq_res_dirac_élec_fin}. Precisely, we set 
    \[\Zpaz:=\lcb\mfrak\in\Mpaz\tq 1+V\Rpaz_0(\mfrak)\rho\mbox{ is not invertible}\rcb\]
    and for $\mfrak\in\Mpaz\setminus\Zpaz$
    \begin{equation}\label{def_R_V}
        \Rpaz(\mfrak):=\Rpaz_0(\mfrak)\lp 1+V\Rpaz_0(\mfrak)\rho\rp^{-1}\lp 1-V \Rpaz_0(\mfrak)(1-\rho)\rp.
    \end{equation}
    This makes sense since both $1-V\Rpaz_0(\mfrak)(1-\rho)$ and $\lp 1+V\Rpaz_0(\mfrak)\rho\rp^{-1}$
    map $\Hfrak_\comp$ into itself. Indeed, for all $\Tilde{\rho}\in\Ccal$ such that $\rho\prec\Tilde{\rho}$ we have
    \begin{equation}\label{2mapL2c}
        \lp 1+V\Rpaz_0(\mfrak)\rho\rp^{-1}\Tilde{\rho}=\Tilde{\rho}\lp 1+V\Rpaz_0(\mfrak)\rho\rp^{-1}.
    \end{equation}
    since 
    \[\lp 1+V\Rpaz_0(\mfrak)\rho\rp\Tilde{\rho}=\Tilde{\rho}\lp 1+V\Rpaz_0(\mfrak)\rho\rp.\]

    \bigskip
    
    If $(z,\omega)\not\in\Zpaz$ and $\Im\omega>0$ then $z\not\in\Sp(H_0)$ and 
    \[1+V(H_0-z)^{-1}\rho=1+V\Rpaz_0(z,\omega)\rho\]
    is invertible. Then $1+V(H_0-z)^{-1}$ is invertible by \eqref{eq_res_dirac_élec_temp2} so $z\not\in\Sp(H)$ by  \eqref{eq_res_dirac_élec_temp1}. Hence \eqref{eq_res_dirac_élec_fin} yield
    \[\Rpaz(z,\omega)=(H-z)^{-1} \mbox{ on } \Hfrak_\comp.\]
    
    \bigskip
    
    We must now show the meromorphy of $\Rpaz$. The game is to apply the Fredholm analytic Theorem \cite[Theorem VI.14]{ReSi72} to the holomorphic function
    \[\mfrak\in\Mpaz\mapsto 1+V\Rpaz_0(\mfrak)\rho\in\Lcal(\Hfrak).\] 
    For this, we must show that $V\Rpaz_0(\mfrak)\rho$ is a compact operator for every $\mfrak\in\Mpaz$ and prove that $\Zpaz\not=\Mpaz$.
    
    Let $\mfrak\in\Mpaz$ and $\tilde{\rho}\in\Ccal$ such that $\rho\prec\tilde{\rho}$. Thanks to Rellich-Kondrachov Theorem, the multiplication by $\tilde{\rho}$ defines a compact operator from $\Dfrak$ to $\Hfrak$. Consequently $\rho \Rpaz_0(\mfrak)\rho=\tilde{\rho}\rho \Rpaz_0(\mfrak)\rho$ is compact on $\Hfrak$ since $\rho \Rpaz_0(\mfrak)\rho$ is bounded from $\Hfrak$ to $\Dfrak$. Since the multiplication by $V$ defines a bounded operator on $\Hfrak$ then $V \Rpaz_0(\mfrak)\rho=V\rho \Rpaz_0(\mfrak)\rho$ is also compact on $\Hfrak$. 
    
    Let $z\in\C$ with $\lv\Im z\rv>\lV V\rV_{\infty}\lV \rho\rV_{\infty}$. Then 
    \[V\Rpaz\lp z,(z^2-1)^{1/2}\rp\rho=V(H_0-z)^{-1}\rho\] 
    and it yields
    \[\lV V\Rpaz\lp z,(z^2-1)^{1/2}\rp\rho\rV_{\Lcal\lp\Hfrak\rp}\leq\frac{\lV V\rV_{\infty}\lV \rho\rV_{\infty}}{\lv\Im z\rv}<1\]
    so that $\lp z,(z^2-1)^{1/2}\rp\in\Mpaz\setminus\Zpaz$ by the Neumann invertibility criterion.
    
    Thanks to the Fredholm analytic Theorem
    \begin{itemize}
        \item $\Zpaz$ is closed and discrete
        \item $\mfrak\mapsto\lp 1+V\Rpaz_0(\mfrak)\rho\rp^{-1}$ is holomorphic from $\Mpaz\setminus\Zpaz$ to $\Lcal(\Hfrak)$
        \item for all $\mfrak\in\Zpaz$ and all $(\Vpaz,\varphi)\in\Acal$ with $\mfrak\in\Vpaz$ there exists a sequence $(B_\ell)_{\ell\geq -L}$ of $\Lcal(\Hfrak)$ with $B_{-L},...,B_{-1}$ of finite rank and $B_{-L}\not=0$ such that for all $\zeta$ in a punctured neighbourhood of $\varphi(\mfrak)$
        \begin{equation}\label{Second}
            \lp 1+V\Rpaz_0(\varphi^{-1}(\zeta))\rho\rp^{-1}=\sum_{\ell=-L}^{+\infty}(\zeta-\varphi(\mfrak))^\ell B_\ell
        \end{equation}
        where the convergence of the series hold in $\Lcal(\Hfrak)$.
    \end{itemize}
    
    Let $\chi,\tilde{\chi}\in\Ccal$. Pick $\Tilde{\rho}\in\Ccal$ such that $\rho,\tilde{\chi}\prec\Tilde{\rho}$. For all $\mfrak\not\in\Zpaz$ one has
    \begin{equation}\label{1mapL2c}
        \lp1-V\Rpaz_0(\mfrak)(1-\rho)\rp\tilde{\chi}=\Tilde{\rho}\lp\tilde{\chi}-V \Rpaz_0(\mfrak)\tilde{\chi}(1-\rho)\rp
    \end{equation}
    therefore \eqref{2mapL2c} yields
    \begin{equation}\label{G+2noms}
        \chi \Rpaz(\mfrak)\tilde{\chi}=\chi \Rpaz_0(\mfrak)\Tilde{\rho}\lp 1+V\Rpaz_0(\mfrak)\rho\rp^{-1}\lp\tilde{\chi}-V\Rpaz_0(\mfrak)\tilde{\chi}(1-\rho)\rp.
    \end{equation}
    In particular $\Rpaz$ is holomorphic from $\Mpaz\setminus\Zpaz$ to $\Lcal(\Hfrak_\comp,\Dfrak_\loc)$.

    Now let $\mfrak\in\Zpaz$ and $(\Vpaz,\varphi)\in\Acal$ with $\mfrak\in\Vpaz$. Thanks to Proposition \ref{HolomorphyR0}, for all $\zeta$ in a neighbourhood of $\varphi(\mfrak)$
    \begin{equation}\label{First}
        \chi \Rpaz_0(\varphi^{-1}(\zeta))\Tilde{\rho}=\sum_{\ell=0}^{+\infty}(\zeta-\varphi(\mfrak))^\ell\chi A_\ell\Tilde{\rho},
    \end{equation}
    \begin{equation}\label{Third}
        \tilde{\chi}-V\Rpaz_0(\varphi^{-1}(\zeta))\tilde{\chi}(1-\rho)=\sum_{\ell=0}^{+\infty}(\zeta-\varphi(\mfrak))^\ell C_\ell\tilde{\chi}
    \end{equation}
    with
    \[C_\ell:=\left\{
    \begin{array}{ll}
        1-VA_0(1-\rho) \mbox{ if } \ell=0 \\
        -VA_\ell(1-\rho) \mbox{ otherwise} 
    \end{array},
    \right.\]
    where \eqref{First} holds in $\Lcal(\Hfrak,\Dfrak)$ and \eqref{Third} holds in $\Lcal(\Hfrak)$. One deduces from \eqref{G+2noms} that for all $\zeta$ in a punctured neighbourhood of $\varphi(\mfrak)$
    \[\chi \Rpaz(\varphi^{-1}(\zeta))\tilde{\chi}=\sum_{\ell=-L}^{+\infty}(\zeta-\varphi(\mfrak))^\ell\sum_{(k,m,n)\in I_\ell}\chi A_k\Tilde{\rho} B_m C_n\tilde{\chi}\]
    with 
    \[I_\ell:=\lcb(k,m,n)\in\N\times\Z\times\N\tq m\geq-L, k+m+n=\ell\rcb\]
    where the convergence of the series hold in $\Lcal(\Hfrak,\Dfrak)$. But \eqref{2mapL2c} and \eqref{1mapL2c} yield
    $B_m\Tilde{\rho}=\Tilde{\rho} B_m$ and $C_n\tilde{\chi}=\Tilde{\rho}C_n\tilde{\chi}$. Since $\chi, \tilde{\chi}$ and $\Tilde{\rho}$ are arbitrary, $B_m$ maps $\Hfrak_\comp$ into itself so
    \[D_\ell:=\sum_{(k,m,n)\in I_\ell}A_kB_m C_n\in\Lcal(\Hfrak_\comp,\Dfrak_\loc)\]
    makes sense and 
    \[\sum_{(k,m,n)\in I_\ell}\chi A_k\Tilde{\rho} B_m C_n\tilde{\chi}=\chi D_\ell\tilde{\chi}.\]
    It is now obvious that the continuation of 
    \[\zeta\mapsto \Rpaz(\varphi^{-1}(\zeta))-\sum_{\ell=1}^{L}\frac{D_{-\ell}}{(\zeta-\varphi(\mfrak))^\ell}\]
    by $D_0$ in $\varphi(\mfrak)$ is holomorphic from a neighbourhood of $\varphi(\mfrak)$ to $\Lcal(\Hfrak_\comp,\Dfrak_\loc)$.

    To conclude that $(\Zpaz,\Rpaz)$ is a finite meromorphic function from $\Mpaz$ to $\Lcal(\Hfrak_\comp,\Dfrak_\loc)$ we shall show that $D_{-L},...,D_{-1}$ are of finite rank. If $\ell\leq -1$ and $(k,m,n)\in I_\ell$ then $m\leq-1$ therefore $B_m$ has a finite rank and so does $D_\ell$.

    \bigskip
    
    We now show that each element of $\Zpaz$ is a pole of $\Rpaz$. Pick $\pfrak\in\Zpaz$. We use the same notations as above. We shall show that $D_{-L}\not=0$.  For all $\mfrak\not\in\Zpaz$
    \begin{equation}\label{res_are_pole}
        \lp 1+V\Rpaz_0(\mfrak)\rho\rp^{-1}=1-V\Rpaz(\mfrak)\rho.
    \end{equation}
    Indeed, with $\lp 1+V\Rpaz_0(\mfrak)\rho\rp\lp 1+V\Rpaz_0(\mfrak)\rho\rp^{-1}=1$ one gets 
    \[\lp 1+V\Rpaz_0(\mfrak)\rho\rp^{-1}=1-V\Rpaz_0(\mfrak)\rho\lp 1+V\Rpaz_0(\mfrak)\rho\rp^{-1}.\]
    But, since $\rho-1=(\rho-1)\lp 1+V\Rpaz_0(\mfrak)\rho\rp$, we have
    \begin{align*}
        \rho\lp 1+V\Rpaz_0(\mfrak)\rho\rp^{-1}&=\rho-1+\lp 1+V\Rpaz_0(\mfrak)\rho\rp^{-1}\\
        &=\lp 1+V\Rpaz_0(\mfrak)\rho\rp^{-1}\lp 1-V\Rpaz_0(\mfrak)(1-\rho)\rp\rho
    \end{align*}
    so \eqref{res_are_pole} follows from \eqref{def_R_V}. Thanks to \eqref{res_are_pole} and the uniqueness of the Laurent expansion at $\varphi(\pfrak)$ we conclude that $B_{-L}=-VD_L\rho$. In particular $D_L\not=0$ since $B_{-L}\not=0$.

    It remains to characterize $\Zpaz$. Take $(\lambda,\kappa)\in\Zpaz$. The operator $1+V\Rpaz_0(\lambda,\kappa)\rho$ is not invertible and then not injective since $V\Rpaz_0(\lambda,\kappa)\rho$ is compact on $\Hfrak$. Pick $\phi\in\Hfrak$ with $\phi\not=0$ such that $\lp1+V\Rpaz_0(\lambda,\kappa)\rho\rp\phi=0$. We set $\psi:=\Rpaz_0(\lambda,\kappa)(\rho\phi)\not=0$ so that 
    \begin{align*}
        (\Dvec+V-\lambda I_4)\psi&=\rho\phi+V\Rpaz_0(\lambda,\kappa)(\rho\phi)\\
        &=\rho\lp1+V\Rpaz_0(\lambda,\kappa)\rho\rp\phi\\
        &=0.
    \end{align*}
    Conversely, take $(\lambda,\kappa)\in\Mpaz$ and assume that there exists $\phi\in\Hfrak_\comp$ such that
    \[\psi:=\Rpaz_0(\lambda,\kappa)\phi\not=0 \mbox{ and } (\Dvec+V)\psi=\lambda\psi.\] 
    Then
    \begin{align*}
        \phi&=(\Dvec-\lambda I_4)(\Rpaz_0(\lambda,\kappa)\phi)\\
        &=-V\Rpaz_0(\lambda,\kappa)\phi.
    \end{align*}
    It follows that $\rho\phi=\phi$ and then $\lp 1+V\Rpaz_0(\lambda,\kappa)\rho\rp\phi=0$. Eventually $\phi\not=0$ since $\psi\not=0$. Hence $(\lambda,\kappa)\in\Zpaz$.
\end{proof}

\subsection{Absence of resonances of $H$} 
This part of the section is dedicated to the proof of Theorem \ref{Rellich} in the context \ref{PotentialContext}, that is the following Theorem.

\begin{theorem}\label{RellichPotential}
    Assume that $V^*=V$. If $(\lambda,\kappa)\in\Mpaz$ with $\kappa\in\R\setminus\lcb 0\rcb$ then $(\lambda,\kappa)$ is not a resonance of $H$.
\end{theorem}

We denote by $\B_r$ (resp. $\Sbb_r$) the ball (resp. sphere) of $\R^3$ whose center is $\nulvec$ and radius is $r>0$. The surface measure on $\Sbb_r$ is denoted $\sigma_r$. 

\begin{proof} 
    One proceeds by contradiction : assume that $(\lambda,\kappa)\in\Mpaz$ is a resonance of $H$ with $\kappa\in\R\setminus\lcb 0\rcb$. Thanks to Theorem \ref{PotentialResolventContinued} one can pick $\phi\in\Hfrak_\comp$ such that $\psi:=\Rpaz_0(\lambda,\kappa)\phi$ is not identically zero and solves the equation 
    \[(\Dvec+V)\psi=\lambda\psi.\]

    \textit{Step 1. We prove that $(\lambda I_4+\beta+\alpha\cdot\xi)\Fcal\phi(\xi)=0$ when $\lv\xi\rv^2=\kappa^2$.} 
    
    Let $r_0>0$ such that $\Supp\phi\subset\B_{r_0}$ and let $r>r_0$. Then  $\psi$ is smooth on $\R^3\setminus\B_{r_0}$ (see Proposition \ref{FreeResolventAsymptotics}). Using $V^*=V$, the fact that $\lambda\in\R$ and an integration by parts, one gets
    \begin{align*}
        \lambda\int_{\B_r}\lv\psi(\xvec)\rv^2d\xvec&=\int_{\B_r}\psi(\xvec)^*(\Dvec+V)\psi(\xvec)d\xvec\\
        &=\int_{\B_r}\lp(\Dvec+V)\psi(\xvec)\rp^*\psi(\xvec)d\xvec+\int_{\Sbb_r}\psi(\svec)^*\lp-\imath\alpha\cdot\frac{\svec}{r}\rp\psi(\svec)d\sigma_r(\svec)
        \\
        &=\lambda\int_{\B_r}\lv\psi(\xvec)\rv^2d\xvec-\imath r^2\int_{\Sbb_1}\psi(r\svec)^*(\alpha\cdot\svec)\psi(r\svec)d\sigma_1(\svec)
    \end{align*}
    and hence
    \begin{equation}\label{boundary_term_zero}
        \int_{\Sbb_1}\psi(r\svec)^*(\alpha\cdot\svec)\psi(r\svec)d\sigma_1(\svec)=0.
    \end{equation}
    For $\svec\in\Sbb_1$, as $\lambda^2-\kappa^2=1$ one has
    \[(\lambda I_4+\beta+\kappa\alpha\cdot\svec)^*(\alpha\cdot\svec)(\lambda I_4+\beta+\kappa\alpha\cdot\svec)=\frac{\kappa}{\lambda}(\lambda I_4+\beta+\kappa\alpha\cdot\svec)^*(\lambda I_4+\beta+\kappa\alpha\cdot\svec)\]
    so Proposition \ref{FreeResolventAsymptotics} yields, as $r\rightarrow+\infty$ 
    \begin{equation}\label{asymptotic_boundary_term}
        \psi(r\svec)^*(\alpha\cdot\svec)\psi(r\svec)=\frac{1}{16\pi^2 r^2}\lp \frac{\kappa}{\lambda}\lv\lp\lambda+\beta+\kappa\alpha\cdot\svec\rp\Fcal\phi\lp\kappa\svec\rp\rv^2+O\lp\frac{1}{r}\rp\rp.
    \end{equation}
    Combining \eqref{boundary_term_zero} and \eqref{asymptotic_boundary_term} and taking the limit $r\rightarrow+\infty$, one obtains
    \[\int_{\Sbb_1}\lv(\lambda I_4+\beta+\kappa\alpha\cdot\svec)\Fcal\phi(\kappa\svec)\rv^2d\sigma_1(\svec)=0\]
    and the conclusion of step 1 follows.
    
    \bigskip
    
    Since Step $1$ holds, we will be able to use the same arguments as for Schrödinger operators (see \cite[Theorem 3.33]{DyZw19}).
    
    \textit{Step 2. We check that $\Rpaz_0(\lambda,\kappa)\phi$ is compactly supported.}
    
    One has $(\Dvec-\lambda I_4)(\Rpaz_0(\lambda,\kappa)\phi)=\phi$ and then in $\Scal'(\R^3)^4$ (as $\lambda^2-\kappa^2=1$)
    \[(-\Delta-\kappa^2I_4)(\Rpaz_0(\lambda,\kappa)\phi)=(\Dvec+\lambda)\phi.\]
    Let $E:\C^3\rightarrow\C^4$ be the entire continuation of $ \Fcal\lp(\Dvec+\lambda)\phi\rp$. For $\xi\in\C^3$ we denote 
    \[\xi^2:=\xi_1^2+\xi_2^2+\xi_3^2\in\C.\]
    Consider the analytic surface                        
    \[\pazocal{A}:=\lcb\xi\in\C^3\tq\xi^2=\kappa^2\rcb.\]
    Thanks to \cite[Theorem 7.3.2]{Ho03}, if we prove that 
    \[\xi\in\C^3\setminus\pazocal{A}\mapsto\frac{E(\xi)}{\xi^2-\kappa^2}\in\C^4\]
    can be continued in an holomorphic function from $\C^3$ to $\C^4$ then $\Rpaz_0(\lambda,\kappa)\phi$ is compactly supported. Even if this result is implicitaly used in previous works on Schrödinger operators, let us give a precise proof.
    
    Let $\xi_0\in\pazocal{A}$. Since $\kappa\not=0$, one of the components of $\xi_0$ is not $\pm\kappa$. Permuting the components if necessary, we can assume that there exists $\theta_0,\varphi_0\in\C$ such that 
    \[\xi_0=\lp \lv\kappa\rv\cos\theta_0\cos\varphi_0, \lv\kappa\rv\sin\theta_0\cos\varphi_0, \lv\kappa\rv\sin\varphi_0 \rp.\]
    Consider the holomorphic function
    \[\Gamma:(r,\theta,\varphi)\in\C^3\mapsto\lp r\cos\theta\cos\varphi , r\sin\theta\cos\varphi , r\sin\varphi\rp\in\C^3.\]
    Then $\Gamma(\lv\kappa\rv,\theta_0,\varphi_0)=\xi_0$ and the jacobian of $\Gamma$ is equal to $\lv\kappa\rv$ at $(\lv\kappa\rv,\theta_0,\varphi_0)$. The holomorphic local inversion Theorem implies that there is an open set $O$ of $\C^3$ centered in $\xi_0$ such that 
    \[(\lv\kappa\rv,\theta_0,\varphi_0)\in O \subset\lcb(r,\theta,\varphi)\in\C^3\tq \Re r>0\rcb\]
    and $\Gamma$ is biholomorphic from $O$ to $\Gamma(O)$. We denote its inverse by
    \[\xi\in\Gamma(O) \mapsto \lp R(\xi) , \Theta(\xi) , \Phi(\xi) \rp\in O.\]
    For $\xi\in\Gamma(O)$ we have $\xi^2=R(\xi)^2$ so $\xi^2\not\in]-\infty,0]$ since $\Re R(\xi)>0$. We can finally write
    \[E(\xi)=h(\xi^2-\kappa^2,\Theta(\xi),\Phi(\xi)) \mbox{ for } \xi\in\Gamma(O)\]
    where 
    \[h(r,\theta,\varphi)=E\lp\Gamma\lp\sqrt{r+\kappa^2},\theta,\varphi\rp\rp \mbox{ for } (r,\theta,\varphi)\in\C^3 \mbox{ with } r\not\in]-\infty,-\kappa^2]\]
    and where $\sqrt{\cdot}$ is the holomorphic branch of the square root defined on $\C\setminus]-\infty,0]$. We see that $h(0,\theta,\varphi)=0$ for all $\theta,\varphi\in \R$ by Step 1 and then for all $\theta,\varphi\in\C$ by the isolated zeros principle applied separately in $\theta$ and $\varphi$. 
    
    For $\xi\in\Gamma(O)$ we get
    \begin{align*}
        E(\xi)&=h\lp\xi^2-\kappa^2,\Theta(\xi),\Phi(\xi)\rp-h\lp 0,\Theta(\xi),\Phi(\xi)\rp\\
        &=\lp\xi^2-\kappa^2\rp\int_0^1\frac{\partial h}{\partial r} \lp t\lp\xi^2-\kappa^2\rp , \Theta(\xi) , \Phi(\xi)\rp dt.
    \end{align*}
    Since the integral term is clearly holomorphic in $\Gamma(O)$, this provides a holomorphic continuation near $\xi_0$ of 
    \[\xi\mapsto\frac{E(\xi)}{\xi^2-\kappa^2}.\]
    Eventually, we deduce that $\Rpaz_0(\lambda,\kappa)\phi$ is compactly supported.
    
    \textit{Step 3. Contradiction.} 
    
    We have just shown in Step 2 that $\psi$ is equal to zero in a neighbourhood of $\infty$. We will now show that $\psi$ is identically zero, which will be our contradiction. For this, we use a kind of Carleman estimate for the Dirac operator : thanks to \cite[Lemma 1]{BeGe87}, for all  real valued function $\varphi\in\Ccal^2(]0,+\infty[)$ and all $u\in\Dfrak$ which support is a compact set of $\R^3\setminus\lcb\nulvec\rcb$
    \[\frac{3}{2}\int_{\R^3\setminus\lcb\nulvec\rcb}e^{2\varphi(\lv\xvec\rv)}\lv(\alpha\cdot D)u(\xvec)\rv ^2d\xvec\geq\int_{\R^3\setminus\lcb\nulvec\rcb}e^{2\varphi(\lv\xvec\rv)}\lp\varphi''(\lv\xvec\rv)+\frac{\varphi'(\lv\xvec\rv)}{\lv\xvec\rv}-\frac{3}{2\lv\xvec\rv^2}\rp\lv u(\xvec)\rv ^2d\xvec.\]
    We set 
    \[\varphi_h(r):=\frac{1}{h}\lp\frac{3}{4}(\log r)^2+\log r\rp \mbox{ for } r,h>0\]
    and apply the Carleman estimate with $u:=\psi(\cdot+\xvec_0)$ where $\xvec_0\not\in\Supp\psi$ is fixed. Then 
    \[\int_{\R^3\setminus\lcb\nulvec\rcb}e^{2\varphi_h(\lv\xvec\rv)}\lv(\lambda I_4-\beta-V(\xvec+\xvec_0))u(\xvec)\rv^2d\xvec\geq\lp\frac{1}{h}-1\rp\int_{\R^3\setminus\lcb\nulvec\rcb}\frac{e^{2\varphi_h(\lv\xvec\rv)}}{\lv\xvec\rv^2}\lv u(\xvec)\rv ^2d\xvec.\]
    Since $\Supp u$ is a compact of $\R^3\setminus\lcb\nulvec\rcb$ one deduces that 
    \[C\int_{\R^3\setminus\lcb\nulvec\rcb}\lv e^{\varphi_h(\lv\xvec\rv)} u(\xvec)\rv ^2d\xvec\geq\lp\frac{1}{h}-1\rp\int_{\R^3\setminus\lcb\nulvec\rcb}\lv e^{\varphi_h(\lv\xvec\rv)}u(\xvec)\rv ^2d\xvec\]
    where $C>0$ does not depend on $h$. By taking $h$ small enough we get
    \[0\geq\int_{\R^3\setminus\lcb\nulvec\rcb}\lv e^{\varphi_h(\lv\xvec\rv)}u(\xvec)\rv ^2d\xvec.\]
    Eventually, $u$ is identically zero and so does $\psi$ : this is a contradiction.  
\end{proof}

\subsection{Resonant states}\label{PotentialResonantState} As explained in Remark \ref{RemarkResonantState}, this part of the section is dedicated to the study of the functions that characterize the existence of a resonance. In particular, we will prove that their set is a finite dimension subspace of $\Dfrak_\loc$. For the whole subsection, we consider $\pfrak\in\Res(H)$ and $\Pi_\pfrak$ the residue at $\pfrak$ of $\Rpaz$ in $(\Vpaz_{\pfrak},\varphi_\pfrak)$. We write $\pfrak=(\lambda,\kappa)$ and assume that $\kappa\not=0$.

\begin{definition} We call
    \begin{itemize}
        \item generalized resonant state associated with $\pfrak$ any function $\psi\in\Ran(\Pi_\pfrak)\setminus\lcb0\rcb$
        \item resonant state associated with $\pfrak$ any function $\psi\in\Ran(\Pi_\pfrak)\setminus\lcb0\rcb$ such that \[(\Dvec+V)\psi=\lambda\psi\]
        \item multiplicity of $\pfrak$ the rank of $\Pi_\pfrak$.
    \end{itemize}
\end{definition}

\begin{proposition}\label{ExplicitLaurentPotential}
    The space $\Ran(\Pi_\pfrak)$ is left invariant by $\Dvec+V$ and
    \[N:\psi\in\Ran(\Pi_\pfrak)\mapsto(\Dvec+V-\lambda)\psi\in\Ran(\Pi_\pfrak)\]
    is nilpotent. Let $L\geq 1$ be the index of $N$. Then
    \[\zeta\mapsto \Rpaz(\varphi_{\pfrak}^{-1}(\zeta))-\sum_{\ell=1}^L\frac{N^{\ell-1}\Pi_\pfrak}{(\zeta-\lambda)^\ell}\]
    has a holomorphic continuation from an open neighbourhood of $\lambda$ to $\Lcal(\Hfrak_\comp,\Dfrak_\loc)$.
\end{proposition}

\begin{remark}
    This result is very similar to the description of the Laurent expansion of the resolvent near an eigenvalue of finite algebraic multiplicity of a general closed operator (see \cite[p. 180-181]{Ka95}).
\end{remark}

\begin{proof}
    Let $T_1=\Pi_\pfrak,T_2,...,T_L\in\Lcal(\Hfrak_\comp,\Dfrak_\loc)$ of finite rank with $T_L\not=0$ such that 
    \[\zeta\mapsto \Rpaz(\varphi_{\pfrak}^{-1}(\zeta))-\sum_{\ell=1}^L\frac{T_\ell}{(\zeta-\lambda)^\ell}\]
    has an holomorphic continuation from an open neighbourhood of $\lambda$ to $\Lcal(\Hfrak_\comp,\Dfrak_\loc)$.
    
    Let $\psi\in\Hfrak_\comp\cap\Dfrak$. Pick $\chi\in\Ccal$ such that $\psi\prec\chi$.  
    For all $(z,\omega)\in\Mpaz\setminus\Zpaz$ 
    \begin{equation}\label{ContinuedResolventEquation}
        \chi \Rpaz(z,\omega)\chi(H-z)\psi=\psi.
    \end{equation}
    Indeed, \eqref{ContinuedResolventEquation} is obviously true if $\Im\omega>0$ and remains true if $\Im\omega\leq 0$ by analytic continuation. Replacing $(z,\omega)$ by $\varphi_\pfrak^{-1}(\zeta)$ in \eqref{ContinuedResolventEquation} and using the uniqueness of the Laurent expansion at $\lambda$, one gets if $\ell<L$
    \[-\chi T_\ell\chi(H-\lambda)\psi+\chi T_{\ell+1}\chi\psi=0\]
    that is 
    \[\chi T_\ell(H-\lambda)\psi=\chi T_{\ell+1}\psi.\]
    Again, $\chi$ being arbitrary, we get
    \[T_\ell(H-\lambda)\psi=T_{\ell+1}\psi.\]
    We deduce that if $l<L$ then
    \[T_{\ell+1}(\Hfrak_\comp\cap\Dfrak)\subset\Ran(T_\ell).\] 
    For the Fréchet topology of $\Dfrak_\loc$, the space $T_{\ell+1}(\Hfrak_\comp\cap\Dfrak)$ is dense in $\Ran(T_{\ell+1})$, but $T_{\ell+1}(\Hfrak_\comp\cap\Dfrak)$ and $\Ran(T_\ell)$ are closed since they are of finite dimensions, hence
    \begin{equation}\label{IncludedRanges}
        \Ran(T_{\ell+1})\subset\Ran(T_\ell).
    \end{equation}
    In a similar fashion one can show that 
    \begin{equation}\label{ExactSequence}
        (\Dvec+V-\lambda)T_\ell=\left\{
        \begin{array}{ll}
            T_{\ell+1} &\mbox{ if } \ell<L \\
            0 &\mbox{ if } \ell=L
        \end{array}.
        \right.
    \end{equation}
    Thanks to \eqref{IncludedRanges} and \eqref{ExactSequence}, one immediately gets that $\Ran(\Pi_\pfrak)$ is left invariant by $\Dvec+V$, $N^{\ell-1}\Pi_\pfrak=T_\ell$ and $N^L=0$.
\end{proof}

\begin{corollary}\label{ResonantIOIOutgoing}
    Let $\psi\in\Dfrak_\loc$. Then $\psi$ is a resonant state associated with $\pfrak$ if and only if $\psi\not=0$, $(\Dvec+V)\psi=\lambda\psi$ and there exists $\phi\in\Hfrak_\comp$ such that $\psi=\Rpaz_0(\lambda,\kappa)\phi$ in a neighbourhood of $\infty$.
\end{corollary}

\begin{proof}
    Assume that $\psi\in\Ran(\Pi_\pfrak)$, $\psi\not=0$ and $(\Dvec+V)\psi=\lambda\psi$. Let $\rho\in\Ccal$ such that $V\prec\rho$. For $z\not\in\Sp(H)$
    \begin{align*}
        (H_0-z)(1-\rho)(H-z)^{-1}&=(H-z)(1-\rho)(H-z)^{-1}\\
        &=-\alpha\cdot(D\rho)(H-z)^{-1}+1-\rho
    \end{align*}
    and then
    \[(1-\rho)\lp H-z\rp^{-1}=-\lp H_0-z\rp^{-1}\alpha\cdot(D\rho)\lp H-z\rp^{-1}+(H_0-z)^{-1}(1-\rho).\]
    By analytic continuation, for all $\chi,\tilde{\chi}\in\Ccal$ and $\mfrak\in\Mpaz\setminus\Zpaz$
    \begin{equation}\label{ResonantToOutgoingResolventPotential}
        (1-\rho)\chi\Rpaz(\mfrak)\tilde{\chi}=-\chi \Rpaz_0(\mfrak)\alpha\cdot(D\rho)\Rpaz(\mfrak)\tilde{\chi}+\chi \Rpaz_0(\mfrak)\tilde{\chi}(1-\rho)
    \end{equation}
    With the notations of Theorem \ref{HolomorphyR0} and Proposition \ref{ExplicitLaurentPotential}, replacing $\mfrak$ by $\varphi_\pfrak^{-1}(\zeta)$ in \eqref{ResonantToOutgoingResolventPotential} and examinating the term of order $-1$ in the Laurent expansion at $\lambda$, one has
    \[(1-\rho)\chi\Pi_\pfrak\tilde{\chi}=\sum_{\ell=0}^{L-1}-\chi A_\ell\alpha\cdot(D\rho)N^\ell\Pi_\pfrak\tilde{\chi}.\]
    Since $\chi$ and $\tilde{\chi}$ are arbitrary we conclude that
    \begin{equation}\label{ResonantToOutgoingResiduePotential}
        (1-\rho)\Pi_\pfrak=-\sum_{\ell=0}^{L-1} A_\ell\alpha\cdot(D\rho)N^\ell\Pi_\pfrak.
    \end{equation}
    Moreover $A_0=\Rpaz_0(\lambda,\kappa)$ and $N^\ell\psi=0$ for $\ell\geq 1$ since $(\Dvec+V)\psi=\lambda\psi$. We conclude that
    \[(1-\rho)\psi=\Rpaz_0(\lambda,\kappa)\lp-\alpha\cdot(D\rho)\psi\rp\]
    With $\phi:=-\alpha\cdot(D\rho)\psi\in\Hfrak_\comp$ we have on $\R^3\setminus\Supp\rho$
    \[\psi=\Rpaz_0(\lambda,\kappa)\phi.\]

    Conversely, assume that $\psi\not=0$, $(\Dvec+V)\psi=\lambda\psi$ and $\psi=\Rpaz_0(\lambda,\kappa)\phi$ on $\R^3\setminus\B_{r_0}$ for some $\phi\in\Hfrak_\comp$ and $r_0>0$. Let $\rho\in\Ccal$ such that $\indic_{\B_{r_0}},\phi\prec\rho$. One has 
    \begin{equation}\label{PsiDecomposition}
        \psi=\rho\psi+(1-\rho)\Rpaz_0(\lambda,\kappa)\phi.
    \end{equation}
    First for $z\not\in\Sp(H)$
    \begin{align*}
        \chi\rho\psi&=\chi(H-z)^{-1}(H-z)\rho\psi\\
        &=\chi(H-z)^{-1}\alpha\cdot(D\rho)\psi+\chi(H-z)^{-1}\rho(\Dvec+V-z)\psi\\
        &=\chi(H-z)^{-1}\alpha\cdot(D\rho)\Rpaz_0(\lambda,\kappa)\phi+(\lambda-z)\chi(H-z)^{-1}\rho\psi
    \end{align*}
    By analytic continuation, for all $\chi\in\Ccal$ and $(z,\omega)\in\Mpaz\setminus\Zpaz$
    \begin{equation}\label{OutgoingToResonantResolventPotential1}
        \chi\rho\psi=\chi\Rpaz(z,\omega)\alpha\cdot(D\rho)\Rpaz_0(\lambda,\kappa)\phi-(z-\lambda)\chi\Rpaz(z,\omega)\rho\psi.
    \end{equation}
    Reversing $\Rpaz_0$ and $\Rpaz$ in the derivation from \eqref{ResonantToOutgoingResolventPotential} and using that $(1-\rho)\phi=0$ we get for all $\chi\in\Ccal$ and $\mfrak\in\Mpaz\setminus\Zpaz$
    \begin{align*}
        (1-\rho)\chi\Rpaz_0(\mfrak)\phi&=-\chi \Rpaz(\mfrak)\alpha\cdot(D\rho)\Rpaz_0(\mfrak)\phi\\
        &=-\chi \Rpaz(\mfrak)\alpha\cdot(D\rho)\lp\Rpaz_0(\mfrak)-\Rpaz_0(\pfrak)\rp\phi-\chi \Rpaz(\mfrak)\alpha\cdot(D\rho)\Rpaz_0(\pfrak)\phi.
    \end{align*}
    With \eqref{OutgoingToResonantResolventPotential1}, this provides for all $\chi\in\Ccal$ and all $\zeta$ in a punctured neighbourhood of $\lambda$
    \begin{align*}
        \chi\rho\psi+(1-\rho)\chi\Rpaz_0(\varphi_\pfrak^{-1}(\zeta))\phi&=-\chi \Rpaz(\varphi_\pfrak^{-1}(\zeta))\alpha\cdot(D\rho)\lp\Rpaz_0(\varphi_\pfrak^{-1}(\zeta))-\Rpaz_0(\varphi_\pfrak^{-1}(\lambda))\rp\phi\\
        &-(\zeta-\lambda)\chi\Rpaz(\varphi_\pfrak^{-1}(\zeta))\rho\psi.
    \end{align*}
    Examinating the term of order $0$ in the Laurent expansion at $\lambda$, with \eqref{PsiDecomposition} one has
    \[\chi\psi=\sum_{\ell=1}^L-\chi N^{\ell-1}\Pi_{\pfrak}\alpha\cdot(D\rho)A_\ell\phi-\chi\Pi_{\pfrak}\rho\psi.\]
    Since $\chi$ is arbitrary, we get
    \[\psi=-\sum_{\ell=1}^L N^{\ell-1}\Pi_{\pfrak}\alpha\cdot(D\rho)A_\ell\phi-\Pi_{\pfrak}\rho\psi\in\Ran(\Pi_\pfrak).\]
\end{proof}

To conclude, we wish to show that the resonant information does not depend on the local chart chosen. 

\begin{corollary}\label{ResonantInformationIndependent}
    Let $(\Vpaz,\varphi)\in\Acal$ such that $\pfrak\in\Vpaz$ and $\Pi$ be the residue at $\pfrak$ of $\Rpaz$ in $(\Vpaz,\varphi)$. Then $\Ran(\Pi)=\Ran(\Pi_\pfrak)$.
\end{corollary}

\begin{proof}
    Set $\Phi:=\varphi\circ\varphi_\pfrak^{-1}$. Recall that $\Phi$ is a biholomorphism. Pick $r>0$ small enough in such a way that the paths  
    \[\gamma:t\in[0,2\pi]\mapsto\varphi(\pfrak)+re^{it}\in\varphi(\Vpaz) \mbox{ and } \Gamma:=\Phi^{-1}\circ\gamma.\]
    encloses only $\varphi(\pfrak)$ and $\lambda$. Let $\chi\in\Ccal$. Then a change of variable and Proposition \ref{ExplicitLaurentPotential} yield
    \begin{align*}
        \chi\Pi\chi&=\int_{\gamma}\chi \Rpaz(\varphi^{-1}(\zeta))\chi d\zeta\\
        &=\int_{\Gamma}\Phi'(\zeta)\chi \Rpaz(\varphi_\pfrak^{-1}(\zeta))\chi d\zeta\\
        &=\Ind_{\Gamma}(\lambda)\sum_{\ell=1}^L\frac{\Phi^{(\ell)}(\lambda)}{(\ell-1)!}\chi N^{\ell-1}\Pi_{\pfrak}\chi
    \end{align*}
    where $\Ind_{\Gamma}(\lambda)$ is the index of $\Gamma$ around $\lambda$.
    Since $\chi$ is arbitrary, one obtains
    \[\Pi=\sum_{\ell=1}^Lc_\ell N^{\ell-1}\Pi_{\pfrak} \mbox{ with } c_\ell=\Ind_{\Gamma}(\lambda)\frac{\Phi^{(\ell)}(\lambda)}{(\ell-1)!}.\]
    In particular $\Ran(\Pi)\subset\Ran(\Pi_\pfrak)$. Since $c_1\not=0$ we can write 
    \[\Pi=c_1(1+\Tilde{N})\Pi_{\pfrak} \mbox{ with } \Tilde{N}=\frac{1}{c_1}\sum_{\ell=2}^Lc_\ell N^{\ell-1}.\]
    Since $\Tilde{N}$ is nilpotent we deduce that $c_1(1+\Tilde{N})$ is an invertible endomorphism of $\Ran(\Pi_\pfrak)$. Therefore $\Pi$ and $\Pi_\pfrak$ have same rank. We conclude that $\Ran(\Pi)=\Ran(\Pi_\pfrak)$.
\end{proof}

\section{Obstacle scattering}\label{ObstacleScattering}
In this section, we consider the context $\ref{PotentialContext}$. The purpose is to obtain similar results as in the previous section. Again, the ideas and techniques used are those of Sjöstrand and Zworski for the abstract \textit{black box frame}. We adapt them to our more concrete Dirac operator.

\bigskip

We recall that we now have $\Hfrak=\leb^2(\Omega)^4$, $\Ccal=\Ccal^\infty_c(\overline{\Omega})$ and  
\[\Dfrak:=\lcb\psi\in\sob^1(\Omega)^4\tq (-\imath\beta\alpha\cdot\nvec)\psi_{\tq\partial\Omega}=\psi_{\tq\partial\Omega}\rcb.\]

\subsection{Construction of the resonances of $H$}
We wish to prove Theorem \ref{ResolventContinued} in the context $\ref{ObstacleContext}$. Let us first state it again.

\begin{theorem}\label{ObstacleResolventContinued}
    There exists a finite meromorphic function from $\Mpaz$ to $\Lcal(\Hfrak_\comp,\Dfrak_\loc)$ denoted $(\Zpaz,\Rpaz)$ such that for all $(z,\omega)\not\in\Zpaz$ with $\Im\omega>0$ one has $z\not\in\Sp(H)$ and 
    \[\Rpaz(z,\omega)=(H-z)^{-1} \mbox{ on } \Hfrak_\comp.\]
\end{theorem}

\begin{definition}
    A pole of $\Rpaz$ is called a resonance of $H$. The set of the resonances of $H$ is denoted $\Res(H)$.
\end{definition}

\begin{remark}
    Theorem \ref{ObstacleResolventContinued} does not provide a characterization of $\Res(H)$ through a generalized eigenvalue problem as in Theorem \ref{PotentialResolventContinued}. Nevertheless, a result in that sense is given in Lemma \ref{ResonantStateExistence}.
\end{remark}

\begin{proof} 
    Let $\rho_0\in\Ccal^\infty_c(\R^3)$ such that $\indic_{\R^3\setminus\Omega}\prec\rho_0$ and $z\not\in\Sp(H)$. The operator $r_\Omega(1-\rho_0)$ maps $\sob^1(\R^3)^4$ into $\Dfrak$ and
    \[H r_\Omega(1-\rho_0)=r_\Omega H_0(1-\rho_0).\]
    Then, by composition with $(H_0-z)^{-1}$, one has
    \[(H-z)L_0(z)=1-\rho_{0\tq\Omega}+K_0(z)\]
    where 
    \[L_0(z):=r_\Omega(1-\rho_0)(H_0-z)^{-1}e_\Omega \mbox{ and } K_0(z):=-r_\Omega\alpha\cdot(D\rho_0)(H_0-z)^{-1}e_\Omega.\] 
    Now, we take $z_0\not\in\Sp(H)$, $\rho_1\in\Ccal^\infty_c(\R^3)$ such that $\rho_0\prec\rho_1$ and we set
    \[L_1:=\rho_{1\tq\Omega}(H-z_0)^{-1}\rho_{0\tq\Omega}\]
    in such a way that
    \[(H-z)L_1=\rho_{0\tq\Omega}+K_1(z)\]
    with 
    \[K_1(z):=\lp\alpha\cdot(D\rho_{1\tq\Omega})+(z_0-z)\rho_{1\tq\Omega}\rp(H-z_0)^{-1}\rho_{0\tq\Omega}.\]
    Setting 
    \[L(z):=L_0(z)+L_1 \mbox{ and } K(z):=K_0(z)+K_1(z)\]
    we eventually get
    \begin{equation}\label{intuition_prol_mit}
        (H-z)L(z)=1+K(z).
    \end{equation}
    Let $\rho_2\in\Ccal^\infty_c(\R^3)$ such that $\rho_1\prec\rho_2$. Then
    \[1+K(z)=\lp 1+K(z)(1-\rho_{2\tq\Omega})\rp\lp 1+K(z)\rho_{2\tq\Omega}\rp\]
    and both operators 
    \[1+K(z)(1-\rho_{2\tq\Omega}) \mbox{ , } 1-K(z)(1-\rho_{2\tq\Omega})\]
    are inverse of one another. When $1+K(z)\rho_{2\tq\Omega}$ is invertible, we eventually get
    \[(H-z)^{-1}=L(z)\lp 1+K(z)\rho_{2\tq\Omega}\rp^{-1}\lp 1-K(z)(1-\rho_{2\tq\Omega})\rp.\]
    Then it is natural to define, for $(z,\omega)\in\Mpaz$
    \[\tilde{L}(z,\omega):=r_\Omega(1-\rho_0)\Rpaz_0(z,\omega)e_\Omega+L_1 \mbox{ , } \tilde{K}(z,\omega):=-r_\Omega\alpha\cdot(D\rho_0)\Rpaz_0(z,\omega)e_\Omega+K_1(z)\]
    and for 
    \[(z,\omega)\not\in\Zpaz:=\lcb\mfrak\in\Mpaz\tq 1+\tilde{K}(\mfrak)\rho_{2\tq\Omega}\mbox{ is not invertible}\rcb,\]
    \[\Rpaz(z,\omega):=\Tilde{L}(z,\omega)\lp 1+\tilde{K}(z,\omega)\rho_{2\tq\Omega}\rp^{-1}\lp1-\tilde{K}(z,\omega)(1-\rho_{2\tq\Omega})\rp.\]

    One recognizes a very similar structure in the definition of $\Rpaz$ as in the previous section (see Equation \eqref{def_R_V}). Specificly $\tilde{L}$ and $\mfrak\mapsto1-\tilde{K}(\mfrak)(1-\rho_{2\tq\Omega})$ 
    have the same analytical and functional properties as $\Rpaz_0$ and $\mfrak\mapsto1-V\Rpaz_0(\mfrak)(1-\rho)$. Consequently the elements of proof of the finite meromorphy of $(\Zpaz,\Rpaz)$ are extremely similar to those of Theorem \ref{PotentialResolventContinued}. We only detail the application of the Fredholm analytic Theorem. 

    Let $\mfrak\in\Mpaz$. Since $(H-z_0)^{-1}$ is bounded from $\Hfrak$ to $\Dfrak$ (see \eqref{EquivalentNormMIT}) then $\tilde{K}(\mfrak)\rho_{2\tq\Omega}$ is compact on $\Hfrak$ thanks to Rellich-Kondrachov Theorem. Let $z\not\in(-\infty,-1]\cup[1,+\infty)$. Then 
    \begin{align*}
        \lV \tilde{K}\lp z,(z^2-1)^{1/2}\rp\rho_{2\tq\Omega}\rV_{\Lcal(\Hfrak)}&\leq C\lp \lV (H_0-z)^{-1}\rV_{\Lcal\lp\leb^2(\R^3)^4\rp}+\lV (H-z)^{-1}\rV_{\Lcal(\Hfrak)}\rp\\
        &\leq \frac{C}{\lv\Im z\rv}
    \end{align*}
    where $C>0$ does not depend on $z$. Choosing $\lv\Im z\rv$ large enough, the Neumann invertibility criterion ensures that $(z,(z^2-1)^{1/2})\in\Mpaz\setminus\Zpaz$. Eventually, one can apply the Fredholm analytic Theorem $\mfrak\mapsto\tilde{K}(\mfrak)\rho_{2\tq\Omega}$.
\end{proof}

\subsection{Absence of resonances of $H$}
This part of the section is dedicated to the proof of Theorem \ref{Rellich} in the context \ref{ObstacleContext}. In order to use a similar strategy as in the proof of Theorem \ref{RellichPotential}, the following Lemma is needed.

\begin{lemma}\label{ResonantStateExistence}
    Let $(\lambda,\kappa)\in\Res(H)$. There exists $\psi\in\Dfrak_\loc$, $\psi\not=0$, such that 
    \begin{enumerate}
        \item $\Dvec\psi=\lambda\psi$
        \item $\psi=\Rpaz_0(\lambda,\kappa)\phi$ in a neighbourhood of $\infty$ for some $\phi\in\leb^2_\comp(\R^3)^4$.
    \end{enumerate}
\end{lemma}

\begin{proof}
    Pick $(\Vpaz,\varphi)\in\Acal$ such that $(\lambda,\kappa)\in\Vpaz$ and $T_1,...,T_L\in\Lcal(\Hfrak_\comp,\Dfrak_\loc)$ with $T_L\not=0$ such that 
    \[\zeta\mapsto \Rpaz(\varphi^{-1}(\zeta))-\sum_{\ell=1}^L\frac{T_\ell}{(\zeta-\zeta_0)^\ell}\]
    can be continued in an holomorphic function from an open neighbourhood of $\zeta_0:=\varphi(\lambda,\kappa)$ to $\Lcal(\Hfrak_\comp,\Dfrak_\loc)$. Let $\chi,\tilde{\chi}\in\Ccal$. By analytic continuation, for all $(z,\omega)\in\Mpaz\setminus\Zpaz$ 
    \[(H-z)\chi \Rpaz(z,\omega)\tilde{\chi}=\alpha\cdot(D\chi)\Rpaz(z,\omega)\tilde{\chi}+\chi\tilde{\chi}.\]
    Let $z$ be the first component of $\varphi^{-1}$. Recall $z$ is holomorphic from $\varphi(\Vpaz)$ to $\C$ and $z(\zeta_0)=\lambda$. For $\zeta$ in a punctured neighbourhood of $\zeta_0$ one has
    \begin{equation}\label{ContinuedResolventEquationObstacle}
        (H-z(\zeta))\chi \Rpaz(\varphi^{-1}(\zeta))\tilde{\chi}=\alpha\cdot(D\chi)\Rpaz(\varphi^{-1}(\zeta))\tilde{\chi}+\chi\tilde{\chi}
    \end{equation}
    so examinating the terms of order $-L$ in the Laurent expansions from \eqref{ContinuedResolventEquationObstacle} at $\zeta_0$ yields
    \[(H-\lambda)\chi T_L\tilde{\chi}=\alpha\cdot(D\chi)T_L\tilde{\chi}\]
    that is 
    \[\chi (\Dvec-\lambda)T_L\tilde{\chi}=0.\]
    Since $\chi,\tilde{\chi}$ are arbitrary we conclude that
    \begin{equation}\label{EigenvalueEquation}
        (\Dvec-\lambda)T_L=0.
    \end{equation}
    Now let $\rho\in\Ccal$ which is equal to $1$ near $\partial\Omega$. For $z\not\in\Sp(H)$
    \begin{align*}
        (H_0-z)e_\Omega(1-\rho)(H-z)^{-1}&=e_\Omega(\Dvec-z)(1-\rho)(H-z)^{-1}\\
        &=-e_\Omega\alpha\cdot(D\rho)(H-z)^{-1}+e_\Omega(1-\rho)
    \end{align*}
    and then
    \[e_\Omega(1-\rho)\lp H-z\rp^{-1}=-\lp H_0-z\rp^{-1}e_\Omega\alpha\cdot(D\rho)\lp H-z\rp^{-1}+(H_0-z)^{-1}e_\Omega(1-\rho).\]
    By analytic continuation, for all $\chi,\tilde{\chi}\in\Ccal^\infty_c(\R^3)$ and $\mfrak\not\in\Zpaz$
    \begin{equation}\label{ResonantToOutgoingResolventObstacle}
        e_\Omega(1-\rho)\chi_{\tq\Omega} \Rpaz(\mfrak)\tilde{\chi}_{\tq\Omega}=-\chi \Rpaz_0(\mfrak)e_\Omega\alpha\cdot(D\rho)\Rpaz(\mfrak)\tilde{\chi}_{\tq\Omega}+\chi \Rpaz_0(\mfrak)\tilde{\chi} e_\Omega(1-\rho).
    \end{equation}
    Again, the term of order $-L$ in the Laurent expansion at $\zeta_0$ yields
    \[e_\Omega(1-\rho)\chi_{\tq\Omega} T_L\tilde{\chi}_{\tq\Omega}=-\chi \Rpaz_0(\lambda,\kappa)e_\Omega\alpha\cdot(D\rho)T_L\tilde{\chi}_{\tq\Omega}.\]
    Since $\chi,\tilde{\chi}$ are arbitrary we conclude that
    \begin{equation}\label{OutgoingEquation}
        e_\Omega(1-\rho)T_L=- \Rpaz_0(\lambda,\kappa)e_\Omega\alpha\cdot(D\rho)T_L.
    \end{equation}
    Pick $\psi\in\Ran(T_L)$ with $\psi\not=0$. Then the first point of the statement is given by \eqref{EigenvalueEquation} whereas the second one is given by \eqref{OutgoingEquation}.
\end{proof}

We can eventually give the proof of 

\begin{theorem}\label{RellichObstacle}
    Let $(\lambda,\kappa)\in\Mpaz$. 
    \begin{itemize}
        \item If $\kappa\in\R\setminus\lcb0\rcb$ and $\Omega$ is connected then $(\lambda,\kappa)$ is not a resonance of $H$.
        \item If $\kappa=0$ then $(\lambda,\kappa)$ is not a resonance of $H$.
    \end{itemize}
\end{theorem}

\begin{proof} 
    One proceeds by contradiction. Let $(\lambda,\kappa)\in\Res(H)$ with $\kappa\in\R$. Consider $\psi$ and $\phi$ as in Proposition \ref{ResonantStateExistence}. Let $r_0>0$ such that $(\R^3\setminus\Omega)\cup\Supp\phi\subset\B_{r_0}$ and $\psi=\Rpaz_0(\lambda,\kappa)\phi$ on $\R^3\setminus\B_{r_0}$. Let $r>r_0$. Then $\psi$ is smooth on $\R^3\setminus\B_{r_0}$. There are two cases to consider. 
    
    \textit{Case $\kappa\in\R\setminus\lcb0\rcb$.}
    Integrating by part and using that $\lambda$ is real, one has
    \begin{align*}
        \lambda\int_{\Omega\cap\B_r}\lv\psi(\xvec)\rv^2d\xvec&=\int_{\Omega\cap\B_r}\psi(\xvec)^*\Dvec\psi(\xvec)d\xvec\\
        &=\int_{\Omega\cap\B_r}\lp\Dvec\psi(\xvec)\rp^*\psi(\xvec)d\xvec+\int_{\Sbb_r}\psi(\svec)^*\lp-\imath\alpha\cdot\frac{\svec}{r}\rp\psi(\svec)d\sigma_r(\svec)
        \\
        &+\int_{\partial\Omega}\psi_{\tq\partial\Omega}(\svec)^*\lp-\imath\alpha\cdot\nvec(\svec)\rp\psi_{\tq\partial\Omega}(\svec)d\sigma(\svec)\\
        &=\lambda\int_{\Omega\cap\B_r}\lv\psi(\xvec)\rv^2d\xvec-\imath r^2\int_{\Sbb_1}\psi(r\svec)^*(\alpha\cdot\svec)\psi(r\svec)d\sigma_1(\svec)\\
        &+\int_{\partial\Omega}\psi_{\tq\partial\Omega}(\svec)^*\lp-\imath\alpha\cdot\nvec(\svec)\rp\psi_{\tq\partial\Omega}(\svec)d\sigma(\svec)
    \end{align*}
    where, to be precise, $\psi_{\tq\partial\Omega}$ denotes the restriction to $\partial\Omega$ of the trace of $\psi_{\tq\Omega\cap\B_{r_0}}\in\sob^1(\Omega\cap\B_{r_0})^4$. But the MIT boundary condition gives
    \begin{align*}
        \psi_{\tq\partial\Omega}^*\lp-\imath\alpha\cdot\nvec\rp\psi_{\tq\partial\Omega}&=\lp B\psi_{\tq\partial\Omega}\rp^*\lp-\imath\alpha\cdot\nvec\rp\psi_{\tq\partial\Omega}\\
        &=\psi_{\tq\partial\Omega}^*(\imath\alpha\cdot\nvec)B\psi_{\tq\partial\Omega}\\
        &=\psi_{\tq\partial\Omega}^*(\imath\alpha\cdot\nvec)\psi_{\tq\partial\Omega}
    \end{align*}
    that is $\psi_{\tq\partial\Omega}^*\lp-\imath\alpha\cdot\nvec\rp\psi_{\tq\partial\Omega}=0$ and eventually 
    \[\int_{\Sbb_1}\psi(r\svec)^*(\alpha\cdot\svec)\psi(r\svec)d\sigma_1(\svec)=0.\]
    From here, one can use step 1 and 2 of the proof of Theorem \ref{RellichPotential} in order to obtain that $\Rpaz_0(\lambda,\kappa)\phi$ has a compact support. Hence, $\psi=0$ in a neighbourhood of $\infty$. Since $\psi\in\sob^1_\loc(\R^3)^4$ and
    \[-\Delta\psi=\kappa^2\psi\]
    we deduce that $\psi\in\sob^2_\loc(\R^3)^4$ and
    \[\lv\Delta\psi(\xvec)\rv=\lv\kappa\rv^2\lv\psi(\xvec)\rv \mbox{ for almost all } \xvec\in\Omega.\]
    Since $\Omega$ is connected, the unique continuation Theorem (see \cite[Theorem XIII.63]{ReSi72}) yields $\psi=0$ : this is a contradiction. 
    
    \textit{Case $\kappa=0$.} First, we develop the square
    \begin{align*}
        \int_{\Omega\cap\B_r}\lv\psi\rv^2&=\int_{\Omega\cap\B_r}\lv\Dvec\psi\rv^2\\
        &=\int_{\Omega\cap\B_r}\lv\psi\rv^2+\int_{\Omega\cap\B_r}\lv(\alpha\cdot D)\psi\rv^2+2\Re\int_{\Omega\cap\B_r}(\beta\psi)^*(\alpha\cdot D)\psi
    \end{align*}
    to get
    \[\int_{\Omega\cap\B_r}\lv(\alpha\cdot D)\psi\rv^2=-2\Re\int_{\Omega\cap\B_r}(\beta\psi)^*(\alpha\cdot D)\psi.\]
    Then an integration by parts and the MIT boundary condition yield 
    \begin{align*}
        \int_{\Omega\cap\B_r}\lp\beta\psi(\xvec)\rp^*(\alpha\cdot D)\psi(\xvec)d\xvec&=-\int_{\Omega\cap\B_r}\lp(\alpha\cdot D)\psi(\xvec)\rp^*\beta\psi(\xvec)d\xvec\\
        &+\int_{\Sbb_r}\psi(\svec)^*\beta\lp-\imath\alpha\cdot\frac{\svec}{r}\rp\psi(\svec)d\sigma_r(\svec)\\
        &+\int_{\partial\Omega}\psi_{\tq\partial\Omega}(\svec)^*\beta\lp-\imath\alpha\cdot\nvec(\svec)\rp\psi_{\tq\partial\Omega}(\svec)d\sigma(\svec)\\
        &=-\int_{\Omega\cap\B_r}\lp(\alpha\cdot D)\psi(\xvec)\rp^*\beta\psi(\xvec)d\xvec\\
        &-ir^2\int_{\Sbb_1}\psi(r\svec)^*\beta(\alpha\cdot\svec)\psi(r\svec)d\sigma_1(\svec)+\int_{\partial\Omega}\lv\psi_{\tq\partial\Omega}(\svec)\rv^2d\sigma(\svec)
    \end{align*}
    We end up with 
    \[\int_{\Omega\cap\B_r}\lv(\alpha\cdot D)\psi(\xvec)\rv^2d\xvec+\int_{\partial\Omega}\lv\psi_{\tq\partial\Omega}(\svec)\rv^2d\sigma(\svec)=ir^2\int_{\Sbb_1}\psi(r\svec)^*\beta(\alpha\cdot\svec)\psi(r\svec)d\sigma_1(\svec).\]
    Applying Proposition \ref{FreeResolventAsymptotics} with $\kappa=0$ and $\lambda=\pm1$ and since
    \[(\pm I_4+\beta)^*\beta(\alpha\cdot\svec)(\pm I_4+\beta)=0\]
    for $\svec\in\Sbb_1$ then as $r\rightarrow+\infty$ 
    \[\psi(r\svec)^*(\beta\alpha\cdot\svec)\psi(r\svec)=O\lp\frac{1}{r^3}\rp.\]
    We end up with
    \begin{equation}\label{asymptotic_alphaD}
        \int_{\Omega\cap\B_r}\lv(\alpha\cdot D)\psi(\xvec)\rv^2d\xvec=-\int_{\partial\Omega}\lv\psi_{\tq\partial\Omega}(\svec)\rv^2d\sigma(\svec)+O\lp\frac{1}{r}\rp
    \end{equation}
    Since the l.h.s from \eqref{asymptotic_alphaD} is non-negative, we get $\psi_{\tq\partial\Omega}=0$. It remains 
    \[\int_{\Omega\cap\B_r}\lv(\alpha\cdot D)\psi(\xvec)\rv^2d\xvec=O\lp\frac{1}{r}\rp\]
    and hence $(\alpha\cdot D)\psi$ is identically zero. Pick $\svec\in\Sbb_1$. Thanks to Proposition \ref{FreeResolventAsymptotics} we have 
    \[(\alpha\cdot D)\psi(r\svec)=\frac{\imath}{4\pi r^2}(\alpha\cdot\svec)(\pm I_4+\beta)\int_{\R^3}\phi+O\lp\frac{1}{r^3}\rp\]
    and hence 
    \begin{equation}\label{coeff_nul_asymptotic}
        (\pm I_4+\beta)\int_{\R^3}\phi=\nulvec
    \end{equation}
    since $(\alpha\cdot D)\psi=0$ and as $\alpha\cdot\svec$ is an invertible matrix. To conclude, we must notice that the l.h.s from \eqref{coeff_nul_asymptotic} is the first coefficient of the asymptotics of Proposition \ref{FreeResolventAsymptotics}. This yields, for all $j\in\llbracket 1,3\rrbracket$
    \[\psi(r\svec),D_j\psi(r\svec)=O\lp\frac{1}{r^2}\rp.\]
    In short $\psi\in\sob^1_0(\Omega)^4$ and $(\alpha\cdot D)\psi=0$. Let $\Psi\in\sob^1(\R^3)^4$ be the continuation of $\psi$ by zero outside $\Omega$. Then $(\alpha\cdot D)\Psi=0$. Finally, for almost all $\xi\in\R^3$, $(\alpha\cdot \xi)\Fcal\Psi(\xi)=0$ and then $\Fcal\Psi(\xi)=0$. Hence $\Psi$ is identically zero and so does $\psi$ : this is a contradiction.
\end{proof}

\subsection{Resonant states}\label{ObstacleResonantState} We wish to give a complement to Lemma \ref{ResonantStateExistence} (as in Subsection \ref{PotentialResonantState}). We consider $\pfrak=(\lambda,\kappa)\in\Res(H)$ and $\Pi_\pfrak$ the residue at $\pfrak$ of $\Rpaz$ in $(\Vpaz_{\pfrak},\varphi_\pfrak)$. Thanks to Theorem \ref{RellichObstacle} one has $\kappa\not=0$. We keep the same definitions as in the potential's case.

\begin{definition} We call
    \begin{itemize}
        \item generalized resonant state associated with $\pfrak$ any function $\psi\in\Ran(\Pi_\pfrak)\setminus\lcb0\rcb$
        \item resonant state associated with $\pfrak$ any function $\psi\in\Ran(\Pi_\pfrak)\setminus\lcb0\rcb$ such that \[\Dvec\psi=\lambda\psi\]
        \item multiplicity of $\pfrak$ the rank of $\Pi_\pfrak$.
    \end{itemize}
\end{definition}

Next Proposition is proven in the same way as Proposition \ref{ExplicitLaurentPotential} (it is based on standard resolvent equations).  

\begin{proposition}\label{ExplicitLaurentObstacle}
    The space $\Ran(\Pi_\pfrak)$ is left invariant by $\Dvec$ and
    \[N:\psi\in\Ran(\Pi_\pfrak)\mapsto(\Dvec-\lambda)\psi\in\Ran(\Pi_\pfrak)\]
    is nilpotent. Let $L\geq 1$ be the index of $N$. Then
    \[\zeta\mapsto \Rpaz(\varphi_{\pfrak}^{-1}(\zeta))-\sum_{\ell=1}^L\frac{N^{\ell-1}\Pi_\pfrak}{(\zeta-\lambda)^\ell}\]
    has a holomorphic continuation from an open neighbourhood of $\lambda$ to $\Lcal(\Hfrak_\comp,\Dfrak_\loc)$.
\end{proposition}

In particuliar, Corollary \ref{ResonantInformationIndependent} still holds for the obstacle's case. Then we deduce the following Corollary.

\begin{corollary}\label{CharaterizationResonantState}
    Let $\psi\in\Dfrak_\loc$. Then $\psi$ is a resonant state associated with $\pfrak$ if and only if $\psi\not=0$, $\Dvec\psi=\lambda\psi$ and there exists $\phi\in\leb^2_\comp(\R^3)^4$ such that $\psi=\Rpaz_0(\lambda,\kappa)\phi$ in a neighbourhood of $\infty$.
\end{corollary}

\begin{remark}
    \begin{enumerate}
        \item The elements of proof of Corollary \ref{CharaterizationResonantState} are again very similar to those of Corollary \ref{ResonantIOIOutgoing}. To be precise, the analoguous from \eqref{ResonantToOutgoingResolventPotential} is \eqref{ResonantToOutgoingResolventObstacle}. 
        \item Corollary \ref{CharaterizationResonantState} provides a reciprocal to Lemma \ref{ResonantStateExistence}.
    \end{enumerate}
\end{remark}

\appendix

\section{Regularity of the MIT bag model}\label{RegularityMIT}
In this appendix, we wish to prove \eqref{EquivalentNormMIT}. The lower bound is trivial. The upper bound will follow from the explicit computation of the quadratic form of the square of $H$ (see \cite[Theorem 1.5]{ArLeTRa17}). This is very related to the Schrödinger-Lichnerowicz formula (see \cite{HiMoRo02}). For the sake of precision and to remove any ambiguity, we redo the proofs. In particular, we insist on a density argument and distinguish the derivatives on $\Omega$ and $\partial\Omega$. 

\bigskip

For the reading, we set $B:=-\imath\beta\alpha\cdot\nvec$ in such a way that
\[\Dfrak:=\lcb\psi\in\sob^1(\Omega)^4\tq B\psi_{\tq\partial\Omega}=\psi_{\tq\partial\Omega}\rcb.\]

\subsection{A density result} 
The following Lemma will be used in the next section.

\begin{lemma}\label{density_result}
    The set $\Dfrak\cap\Ccal^1_c(\overline{\Omega})^4$ is dense in $\Dfrak$ for $\lV\cdot\rV_{\sob^1(\Omega)^4}$.
\end{lemma}

\begin{proof}
    Consider $\Nvec\in\Ccal^1_c(\R^3)^3$ such that $\Nvec_{|\partial\Omega}=\nvec$. Take $\Upaz$ an open subset of $\R^3$ containing $\partial\Omega$ and such that $\Nvec(\xvec)\not=\nulvec$ for $\xvec\in \Upaz$. Pick a function $\chi\in\Ccal^\infty_c(\R^3)$ such that $\Supp\chi\subset\Upaz$ and $\indic_{\partial\Omega}\prec\chi$. We define the matrix-valued functions $Q^\pm$ by
    \[Q^\pm(\xvec):=\left\{
        \begin{array}{ll}
            \frac{\chi(\xvec)}{2}\lp I_4\pm\imath\beta\alpha\cdot \frac{\Nvec(\xvec)}{\lv \Nvec(\xvec)\rv}\rp \mbox{ if } \xvec\in\Upaz \\
            0 \mbox{ otherwise.} 
        \end{array}
    \right.\]
    Clearly $Q^\pm$ are continuously differentiable and compactly supported. Let $\psi\in\Dfrak$. Then
    \begin{align*}
        \lp Q^+_{|\Omega}\psi+(1-\chi_{|\Omega})\psi\rp_{\tq\partial\Omega}&=\frac{I_4-B}{2}\psi_{\tq\partial\Omega}\\
        &=0.
    \end{align*}
    Hence $Q^+_{|\Omega}\psi+(1-\chi_{|\Omega})\psi\in\sob^1_0(\Omega)^4$ and we can pick a sequence $(g_k)_{k\geq 0}$ of functions of $\Ccal^\infty_c(\R^3)^4$ all supported in $\Omega$ such that 
    \[\lim_{k\rightarrow+\infty}g_{k|\Omega}=Q^+_{|\Omega}\psi+(1-\chi_{|\Omega})\psi \mbox{ in }\sob^1(\Omega)^4.\]
    Consider also a sequence $(h_k)_{k\geq 0}$ of $\Ccal^\infty_c(\R^3)^4$ such that 
    \[\lim_{k\rightarrow+\infty}h_{k|\Omega}=\psi \mbox{ in }\sob^1(\Omega)^4.\]
    Now set 
    \[f_k:=g_k+Q^-_{\tq\Omega}h_k\in\Ccal^1_c(\R^3)^4.\]
    Since $B^2=I_4$ we get
    \begin{align*}
        B f_{k\tq\partial\Omega}&=B\frac{I_4+B}{2}h_{k\tq\partial\Omega}\\
        &=\frac{B+I_4}{2}h_{k\tq\partial\Omega}\\
        &=f_{k\tq\partial\Omega}
    \end{align*}
    and then $f_k\in\Dfrak$. Since $Q^++Q^-=\chi I_4$ then one has in $\sob^1(\Omega)^4$
    \begin{align*}
        \lim_{k\rightarrow+\infty}f_{k|\Omega}&=Q^+_{|\Omega}\psi+(1-\chi_{|\Omega})\psi+Q^-_{|\Omega}\psi\\
        &=\psi.
    \end{align*}    
\end{proof}

\subsection{Schrödinger-Lichnerowicz formula}
In this section, we compute the quadratic form of the square of $H$. Let us quickly recall some points about the Weingarten map $W(\svec)$, $\svec\in\partial\Omega$. One has to know that $W(\svec)$ is a symmetric endomorphism of the tangent space $T_\svec\partial\Omega$ (endowed with the scalar product of $\R^3$) such that 
\[W(\svec)\vvec=(\nvec\circ c)'(0)\]
for any curve $c:I\rightarrow\partial\Omega$ ($I$ interval containing $0$) such that $c(0)=\svec$ and $c'(0)=\vvec$. It is then known that the mean curvature
\[h:\svec\in\partial\Omega\mapsto\frac{\Tr(W(\svec))}{2}\in\R\]
is essentially bounded. 

\begin{theorem}[Schrödinger-Lichnerowicz]\label{schrödinger_lichnerowicz}
    For all $\psi\in\Dfrak$
    \begin{equation}\label{square_form_dirac}
        \lV H\psi\rV_{\leb^2(\Omega)^4}^2=\lV\psi\rV_{\sob^1(\Omega)^4}^2+\la \psi_{\tq\partial\Omega},h\psi_{\tq\partial\Omega}\ra_{\leb^2(\partial\Omega)^4}.
    \end{equation}
\end{theorem}

\begin{proof} 
    Let $f\in\Ccal^\infty_c(\R^3)^4$. Using that $\beta$ and the $\alpha_j$'s anti-commute and are unitary and hermitian, one gets by developping
    \begin{align*}
        \lv(\beta+\alpha\cdot D)f\rv^2&=\lv f\rv^2+\lv D_1f\rv^2+\lv D_2f\rv^2+\lv D_3f\rv^2+\partial_1(f^*\alpha_1\alpha_2\partial_2f+f^*\alpha_1\alpha_3\partial_3f)\\
        &+\partial_2(f^*\alpha_2\alpha_3\partial_3f-f^*\alpha_1\alpha_2\partial_1f)-\partial_3(f^*\alpha_1\alpha_3\partial_1f+f^*\alpha_2\alpha_3\partial_2f)
    \end{align*}
    so by Stokes Theorem
    \begin{equation}\label{ipp}
        \int_\Omega\lv(\beta+\alpha\cdot D)f(\xvec)\rv^2d\xvec=\int_\Omega\lv f(\xvec)\rv^2d\xvec+\int_\Omega\lv Df(\xvec)\rv^2d\xvec+\int_{\partial\Omega}f(\svec)^*Pf(\svec)d\sigma(\svec)
    \end{equation}
    where 
    \begin{equation}\label{defP}
        Pf:=\alpha_1\alpha_2(n_1\partial_2f-n_2\partial_1f)+\alpha_1\alpha_3(n_1\partial_3f-n_3\partial_1f)+\alpha_2\alpha_3(n_2\partial_3f-n_3\partial_2f)
    \end{equation}
    
    Now take $f\in\Ccal^1_c(\R^3)^4$ and a sequence $(\rho_n)_{n\geq 0}$ of $\Ccal$ which is an approximation of identity. Consider the sequence $(f_n:=\rho_n*f)_{n\geq 0}$ of $\Ccal^\infty_c(\R^3)^4$. Then $(f_n)_{n\geq 0}$ and its first derivatives converge to $f$ and its first derivatives uniformly and in $\Hfrak$. In particular, as $\partial\Omega$ is compact, \eqref{ipp} remains true with $f\in\Ccal^1_c(\R^3)^4$.
    
    Until the end of the proof, we consider $f\in\Dfrak\cap\Ccal^1_c(\overline{\Omega})$. We must compute the surface integral from \eqref{ipp} using the boundary condition. In fact, we will be able to make the simplifications pointwise. Hence, we fix $\svec\in\partial\Omega$ and we drop the dependance of the functions in $\svec$ to allege the reading. In \eqref{defP}, one can recognize the structure of an exterior product by the normale : we are only considering tangential derivatives. Since $W=W(\svec)$ is symmetric we can consider $(\tvec,\tvec')$ an orthonormal basis of $T_\svec\partial\Omega$ such that $W\tvec=\lambda\tvec$ and $W\tvec'=\lambda'\tvec'$ ($\lambda,\lambda'\in\R$ are the principal curvatures of $\partial\Omega$) and such that $(\tvec,\tvec',\nvec)$ is a direct orthonormal basis of $\R^3$. For $\vvec\in\R^3$ we define $\partial_\vvec f\in\C^4$ by
    \[(\partial_{\vvec}f)_k:=\vvec\cdot\nabla f_k \mbox{ for } k=1,2,3,4.\]
    Since $(\nvec,\tvec,\tvec')$ is orthonormal, one has 
    \[\partial_jf=n_j\partial_\nvec f+t_j\partial_\tvec f+t'_j\partial_{\tvec'}f \mbox{ for } j=1,2,3.\]
    so \eqref{defP} becomes
    \begin{align*}
        Pf&=\lb\alpha_1\alpha_2(\nvec\times\tvec)_3-\alpha_1\alpha_3(\nvec\times\tvec)_2+\alpha_2\alpha_3(\nvec\times\tvec)_1\rb\partial_{\tvec}f\\
        &+\lb\alpha_1\alpha_2(\nvec\times\tvec')_3-\alpha_1\alpha_3(\nvec\times\tvec')_2+\alpha_2\alpha_3(\nvec\times\tvec')_1\rb\partial_{\tvec'}f
    \end{align*}
    which, with the choice of the orientation, simplifies as 
    \begin{equation}\label{Psurface}
        Pf=\imath\gamma_5\lb(\alpha\cdot\tvec')\partial_{\tvec}f-(\alpha\cdot\tvec)\partial_{\tvec'}f\rb
    \end{equation}
    where $\gamma_5:=-\imath\alpha_1\alpha_2\alpha_3$. Pick a curve $c:I\rightarrow\partial\Omega$ such that $c(0)=\svec$ and $c'(0)=\tvec$. Applying the Leibniz formula to the boundary condition seen along the curve $c$ and using that $\partial_\tvec f=(f\circ c)'(0)$ and $(\nvec\circ c)'(0)=W\tvec=\lambda\tvec$ one gets
    \[\partial_\tvec f=-\imath\lambda\beta(\alpha\cdot\tvec)f+B\partial_\tvec f.\]
    Since $(\alpha\cdot X)(\alpha\cdot Y)=X\cdot Y+\imath\gamma_5\alpha\cdot(X\times Y)$ and $\beta\gamma_5=-\gamma_5\beta$ one eventually gets by using again the boundary condition
    \begin{equation}\label{Psurface1}
        (\alpha\cdot\tvec')\partial_\tvec f=-\imath\lambda\gamma_5 f+B(\alpha\cdot\tvec')\partial_\tvec f.
    \end{equation}
    Similary
    \begin{equation}\label{Psurface2}
        (\alpha\cdot\tvec)\partial_{\tvec'} f=\imath\lambda'\gamma_5 f+B(\alpha\cdot\tvec)\partial_{\tvec'} f.
    \end{equation}
    Using $\gamma_5^2=-I_4$, $\gamma_5(\alpha\cdot X)=(\alpha\cdot X)\gamma_5$, $\lambda+\lambda'=2h$ and combining \eqref{Psurface},\eqref{Psurface1} and \eqref{Psurface2} we end up with
    \[Pf=2hf-BPf.\]
    Using a last time the boundary condition, we get 
    \[f(\svec)^*Pf(\svec)=h(\svec)\lv f(\svec)\rv^2.\]
    Hence, Theorem \ref{schrödinger_lichnerowicz} is proven for smooth enough functions. Lemma \ref{density_result} finishes the job.
\end{proof}

\subsection{Conclusion}

\begin{corollary}
    For all $\psi\in\Dfrak$
    \[\lV\psi\rV_{\sob^1(\Omega)^4}\leq C\lV H\psi\rV_{\leb^2(\Omega)^4}\]
    where $C>0$ does not depend on $\psi$.    
\end{corollary}

\begin{proof}
    Pick $\varepsilon>0$. There exists $C_0>0$ such that for all $\psi\in\sob^1(\Omega)^4$
    \[\lV \psi_{\tq\partial\Omega}\rV_{\leb^2(\partial\Omega)^4}^2\leq\varepsilon\lV \psi\rV_{\sob^1(\Omega)^4}^2+C_0\lV \psi\rV_{\leb^2(\Omega)^4}^2.\]
    Take $\psi\in\Dfrak$. Thanks to Theorem \ref{schrödinger_lichnerowicz} and the fact that $h\in\leb^\infty(\partial\Omega)$ one obtains
    \begin{equation}\label{Conclusion1}
        \lV H\psi\rV_{\leb^2(\Omega)^4}^2\geq\lp1-\varepsilon\lV h\rV_{\leb^\infty(\partial\Omega)}\rp\lV \psi\rV_{\sob^1(\Omega)}^2-C_0\lV h\rV_{\leb^\infty(\partial\Omega)}\lV\psi\rV_{\leb^2(\Omega)^4}^2.
    \end{equation}
    Since
    \[\lV H\psi\rV_{\leb^2(\Omega)^4}^2=\lV(\alpha\cdot D)\psi\rV_{\leb^2(\Omega)^4}^2+\lV\psi\rV_{\leb^2(\Omega)^4}^2\]
    then
    \begin{equation}\label{Conclusion2}
        \lV H\psi\rV_{\leb^2(\Omega)^4}^2+C_0\lV h\rV_{\leb^\infty(\partial\Omega)}\lV\psi\rV_{\leb^2(\Omega)^4}^2\leq (1+C_0\lV h\rV_{\leb^\infty(\partial\Omega)})\lV H\psi\rV_{\leb^2(\Omega)^4}^2.
    \end{equation}
    Taking $\varepsilon>0$ small enough in \eqref{Conclusion1} and using \eqref{Conclusion2}, we get the result expected.
\end{proof}

\printbibliography

\end{document}